\documentclass{article}
\usepackage{amsmath,color,amsfonts,amsthm}
\usepackage{amssymb,enumerate,enumitem,verbatim}
\usepackage[dvipsnames]{xcolor}
\usepackage{hyperref,pagecolor}

\newtheorem{lemma}{Lemma}[section]
\newtheorem{corollary}[lemma]{Corollary}
\newtheorem{conjecture}[lemma]{Conjecture}
\newtheorem{claim}[lemma]{Claim}
\newtheorem{theorem}[lemma]{Theorem}
\newtheorem{prop}[lemma]{Proposition}

\theoremstyle{definition}
\newtheorem{defn}[lemma]{Definition}

\global\long\def\CC{C}

\global\long\def\XX{B}
\global\long\def\E{\mathbb{E}}
\global\long\def\N{\mathbb{N}}
\global\long\def\e{\varepsilon}
\global\long\def\eps{\varepsilon}
\global\long\def\P{\mathbb{P}}

\definecolor{darkred}{rgb}{0.6,0,0}

\date{}\label{key}

\title{\vspace{-1.0cm} Embedding rainbow trees with applications to graph labelling and decomposition}

\author{R. Montgomery\thanks{Trinity College, Cambridge, UK, CB2 1TQ. r.h.montgomery@dpmms.cam.ac.uk}
, A. Pokrovskiy\thanks{Department of Mathematics, ETH, 8092 Zurich, Switzerland.dr.alexey.pokrovskiy@gmail.com}
, and B. Sudakov\thanks{Department of Mathematics, ETH, 8092 Zurich, Switzerland.benjamin.sudakov@math.ethz.ch.
Research supported in part by SNSF grant 200021-175573.}
}

\begin{document}

\maketitle

\begin{abstract}
A subgraph of an edge-coloured graph is called rainbow if all its edges have distinct colours. The study of rainbow subgraphs goes back more than two hundred years to the work of Euler on Latin squares. Since then rainbow structures have been the focus of extensive research and have found applications in the areas of graph labelling and decomposition.
An edge-colouring is locally $k$-bounded if each vertex is contained in at most $k$ edges of the same colour. In this paper we prove
that any such edge-colouring of the complete graph $K_n$ contains a rainbow copy of every tree with at most $(1-o(1))n/k$ vertices. As a locally $k$-bounded edge-colouring  of  $K_n$ may have only $(n-1)/k$ distinct colours, this is essentially tight.

As a corollary of this result we obtain asymptotic versions of two long-standing conjectures in graph theory.
Firstly, we prove an asymptotic version of Ringel's conjecture from 1963, showing that any $n$-edge tree packs into the complete graph $K_{2n+o(n)}$ to cover all but  $o(n^2)$ of its edges.
Secondly, we show that all trees have an almost-harmonious labelling. The existence of such a labelling was conjectured by Graham and Sloane in 1980. We also discuss some additional applications.
\end{abstract}

\section{Introduction}\label{intro}

In this paper, we study the appearance of large rainbow trees in edge-coloured complete graphs and present applications of our result to several old
open problems in graph theory.

A \emph{rainbow} subgraph of an edge-coloured graph is one whose edges have different colours.
The study of rainbow subgraphs of edge-coloured graphs goes back more than two hundred years to the work of Euler on Latin squares.
A \emph{Latin square of order $n$} is an $ n \times n$  array filled with $n$ symbols such that each symbol appears once in every row and column.
A \emph{partial transversal} is a collection of cells of the Latin square which do not share the same row, column or symbol.
Starting with Euler (see \cite{Euler, keedwell2015latin}), transversals in Latin squares have been extensively studied.  The most famous open problem in this area is
the  Ryser-Brualdi-Stein Conjecture (see \cite{Brualdi, Stein, WanlessSurvey}), which says that every Latin square has a partial transversal of order $n-1$ and a full transversal (a partial transversal of order $n$) if $n$ is odd.
To every Latin square one can assign an edge-colouring of the complete bipartite graph $K_{n,n}$ by colouring the edge $ij$ by the symbol in the cell $(i,j)$.
This is a \emph{proper colouring}, i.e., one in which any edges which share a vertex have distinct colours. Identifying the cell $(i,j)$ with the edge $ij$, a
partial transversal corresponds to a rainbow matching. Thus, finding transversals is a special case of finding rainbow subgraphs.
Another reason to study rainbow subgraphs  arises in Ramsey theory, more precisely in the canonical version of Ramsey's theorem proved by Erd\H{o}s and Rado \cite{ErdosRado}. Here the goal is to show that locally-bounded edge-colourings of the complete graph $K_n$ contain rainbow copies of certain graphs. An edge-colouring is \emph{locally $k$-bounded} if each vertex is in at most $k$ edges of any one colour.

The most natural problem in the study of rainbow structures is to determine which graphs are guaranteed to have a rainbow copy in any properly coloured complete graph $K_n$. As, when $n$ is even, $K_n$ can be $(n-1)$-edge-coloured, we may ask in general only for rainbow subgraphs with at most $n-1$ edges. This leads to a natural question: which $n$-vertex trees have a rainbow copy in any proper colouring of $K_n$?
Hahn~\cite{hahn1980jeu} conjectured in 1980 the special case that there would always be a rainbow copy of the $n$-vertex path. Disproving this
conjecture, Maamoun and Meyniel~\cite{maamoun1984problem} constructed a proper colouring of $K_n$ with no rainbow Hamilton path. Nevertheless, it is widely believed that any properly coloured $K_n$ contains a rainbow path covering all but exceptionally few vertices. In particular, Andersen~\cite{andersen1989hamilton} in 1989 conjectured that one can always find a rainbow path covering $n-1$ vertices. The progress on this conjecture was initially slow, despite the efforts of various researchers, see, for example, \cite{akbari2007rainbow, gyarfas2010rainbow, gyarfas2011long, gebauer2012rainbow, chen2015long}.
Only recently did Alon, Pokrovskiy and Sudakov~\cite{alon2016random} show that any properly coloured $K_n$ contains a rainbow path with $n-O(n^{3/4})$ vertices (as part of a rainbow cycle).

Turning to more general trees, there are no previous general results to show that rainbow copies of large trees can be found in properly coloured complete graphs. The example of Maamoun and Meyniel~\cite{maamoun1984problem} can be extended to show that there are proper colourings of $K_n$ which do not contain rainbow copies of certain $n$-vertex trees (see \cite{BenzingTrees}). However, it is still possible that nearly-spanning trees exist in all proper colourings of complete graphs. In this paper, we prove the first result of this type --- we show that any properly coloured $K_n$ contains a rainbow copy of any tree with $n-o(n)$ vertices. We will, in fact, prove the following more general result of which this is a special case. Recall that an edge-colouring of $K_n$ is
\emph{locally $k$-bounded} if each vertex is in at most $k$ edges of any one colour.

\begin{theorem}\label{almostspan}
If $\e,1/k\gg 1/n>0$, then any locally $k$-bounded edge-colouring of $K_n$ contains a rainbow copy of every tree with at most $(1-\e)n/k$ vertices.
\end{theorem}

As a locally $k$-bounded edge-colouring of $K_n$ may have only $(n-1)/k$ distinct colours, Theorem~\ref{almostspan} is tight up to the constant $\e$ for each $k$. Let us note that one of the distinguishing features of our result is that we place no conditions on the trees other than the number of vertices they may have. In comparison,  all of the graph packing and labelling results which we mention below require a maximum degree bound which is conjectured to be unnecessary. The techniques we introduce for embedding trees with high degree vertices may offer new approaches to these problems.

An important tool in our methods is to demonstrate that, in properly coloured complete graphs, a large rainbow matching can typically be found into a random vertex set using a random set of colours chosen with the same density, in fact such a matching can cover almost all the vertices in the random vertex set. This allows most of a large tree to be embedded if it can be decomposed into certain large matchings.
Contrastingly, we need deterministic methods to embed vertices in the tree with high degree. The interplay between the deterministic and random part of the embedding forms a key part of our methods. This is sketched in more detail in Section~\ref{proofsketch}.

\medskip

\noindent
\textbf{Applications to graph decompositions, labellings and orthogonal double-covers.}
Theorem~\ref{almostspan} has applications in different areas of graph theory. We present three applications of it in the areas of graph decompositions, graph labellings, and orthogonal double-covers. With each application we prove an asymptotic form of a well-known conjecture. In each case, we consider some special colouring of $K_n$ (coming from a graph-theoretic problem), and apply Theorem~\ref{almostspan}.

\medskip

\noindent
\textbf{Graph decompositions.}
In graph decompositions one typically asks when the edge set of some graph $G$ can be partitioned into disjoint copies of another graph $H$.
This is a vast topic with many exciting results and conjectures (see, for example, \cite{gallian2009dynamic,wozniak2004packing,yap1988packing}). One of the oldest and best known conjectures in this area, posed by Ringel in 1963 \cite{ringel1963theory},
deals with the decomposition of complete graphs into edge-disjoint copies of a tree.

\begin{conjecture}[Ringel] \label{ringelconj}
Any tree with $n+1$ vertices packs $2n+1$ times into the complete graph $K_{2n+1}$.
\end{conjecture}
This conjecture is known for many very special classes of trees such as caterpillars, trees with $\leq 4$ leaves, firecrackers,  diameter $\leq 5$ trees,   symmetrical trees, trees with $\leq 35$ vertices, and olive trees (see Chapter 2 of \cite{gallian2009dynamic} and the references therein).
There are some partial general results in the direction of Conjecture~\ref{ringelconj}. Typically, for these results, an extensive technical method is developed which is capable of almost-packing any appropriately-sized collection of certain sparse graphs, see,
e.g., \cite{bottcher2016approximate, messuti2016packing, ferber2017packing, kim2016blow}.   In particular, Joos, Kim, K{\"u}hn and Osthus~\cite{joos2016optimal} proved the above conjecture for very large bounded-degree trees. Ferber and Samotij~\cite{ferber2016packing} obtained an almost-perfect packing of almost-spanning trees with maximum degree $O(n/\log n)$.
A different proof of the approximate version of Ringel's conjecture for trees with maximum degree $O(n/\log n)$ was obtained by
Adamaszek, Allen, Grosu, and Hladk{\'y}~\cite{adamaszek2016almost}, using graph labellings. Finally, Allen, B\"ottcher, Hladk{\'y} and Piguet~\cite{allen2017packing} almost-perfectly packed arbitrary spanning graphs with maximum degree $O(n/ \log n)$ and constant degeneracy \footnote{A graph is $d$-degenerate if all its induced subgraphs have a vertex of degree $\leq d$. Trees are exactly the $1$-degenerate, connected graphs.} into large complete graphs.

Here we obtain the first asymptotic solution for this problem applicable with no degree restriction.

\begin{theorem}\label{approxringelkotzig}
	If $\eps \gg 1/n>0$, then any $(n+1)$-vertex tree packs at least $2n+1$ times into the complete graph $K_{(2+\e)n}$.
\end{theorem}

To see the connection between Theorem~\ref{almostspan} and Conjecture~\ref{ringelconj} consider the following edge-colouring of the complete graph with vertex set $\{0,1,\dots,2n\}$, which we call the \emph{near distance (ND-) colouring}. Colour the edge $ij$ by colour $k$, where $k\in [n]$, if either $i=j+k$ or $j=i+k$ with addition modulo $2n+1$.
It is easy to show (as below) that if an $(n+1)$-vertex tree has a rainbow embedding into the ND-colouring of $K_{2n+1}$ then Conjecture~\ref{ringelconj} holds for that tree.

\begin{proof}[Proof of Theorem~\ref{approxringelkotzig}] Let $\ell=(1+\e/3)n$.
	Consider the ND-colouring of the complete graph $K_{2\ell+1}$, defined above. This is a locally 2-bounded colouring of $K_{2\ell+1}$, which thus, by Theorem~\ref{almostspan}, contains a rainbow copy of $T$, $S_0$ say. Now, for each $i\in [\ell]$, let $S_i$ be the tree with vertex set $\{v+i:v\in V(S_0)\}$ and edge set
	$\{\{v+i,w+i\}:vw\in E(S_0)\}$.

	Note that the colour of $vw\in E(K_{2\ell+1})$ is the same as the colour of the edge $\{v+i,w+i\}$ for each $i\in [2\ell]$, and under the translation $x\mapsto x+1$ the edge $vw$ moves around all $2\ell+1$ edges with the same colour. Thus, each tree $S_i$ is rainbow, and all the trees $S_i$ are disjoint.
\end{proof}

\noindent
\textbf{Graph labelling.}
Graph labelling originated in methods introduced by R\'osa~\cite{rosa1966certain} in 1967 as a potential path towards proving Ringel's conjecture. In the intervening decades, a large body of work has steadily developed concerning different labellings and their applications (see \cite{gallian2009dynamic}). One old, well-known, conjecture in this area concerns the \emph{harmonious labelling}. This labelling was introduced by Graham and Sloane~\cite{GS80} and arose naturally out of the study of additive bases. Given an Abelian group $\Gamma$ and a graph $G$, we say that a labelling $f:V(G)\to \Gamma$ is \emph{$\Gamma$-harmonious} if the map $f':E(G)\to \Gamma$ defined by $f'(xy)=f(x)+f(y)$ is injective.
In the case when $\Gamma$ is a group of integers modulo $n$ we omit it from our notation and simply call such a labelling \emph{harmonious}.
In the particular case of an $n$-vertex tree, Graham  and Sloane asked for a harmonious labelling using $\mathbb{Z}_{n-1}$ in which each label is used on some vertex, so that a single label is used on two vertices -- where this exists we call the tree \emph{harmonious}. More generally, we also call a graph with $m$ edges and at most $m$ vertices \emph{harmonious} if it has an injective harmonious labelling with $\mathbb{Z}_m$.
According to an unpublished result of Erd\H{o}s~\cite{GS80}, almost all graphs are not harmonious. On the other hand, Graham and Sloane~\cite{GS80} in 1980 made the following conjecture for trees.

\begin{conjecture}[Graham and Sloane]
All trees are harmonious.
\end{conjecture}
This conjecture is known for many very special classes of trees such as caterpillars, trees with $\leq 31$ vertices, palm trees and fireworks (see Chapter 2 of \cite{gallian2009dynamic} and the references therein).
\.Zak conjectured \cite{zak2009harmonious} an asymptotic weakening of this conjecture --- that every tree has an injective $\mathbb{Z}_{n+o(n)}$-harmonious labelling.

Note that, for any injective labelling of the vertices of the complete graph by elements of an Abelian group, the edge-colouring which is obtained by taking sums of labels of vertices is proper. Therefore we can use  Theorem~\ref{almostspan} to study such colourings. In particular, we can obtain the following general result which shows that every tree is almost harmonious.
\begin{theorem}\label{thmharlabel}
Every $n$-vertex tree $T$ has an injective $\Gamma$-harmonious labelling for any Abelian group $\Gamma$ of order $n+o(n)$.
\end{theorem}
When the group $\Gamma$ is cyclic, this theorem proves \.Zak's conjecture from~\cite{zak2009harmonious}.
\begin{proof}[Proof of Theorem~\ref{thmharlabel}] Suppose $|\Gamma|=\ell= (1+\eps)n$ for some $\eps>0$ and let $T$ be a tree on $n$ vertices. Identify the vertices of $K_\ell$ with the elements of $\Gamma$ and consider an edge-colouring that colours the edge $ij$ by $i+j$. This is a proper colouring, so, when $n$ is large, by Theorem~\ref{almostspan} it contains a rainbow copy of $T$, which corresponds to a harmonious labelling. By taking $\Gamma=\mathbb{Z}_{\ell}$, we deduce that for any $\e>0$ there exists $n_0$ such that any tree with $n\geq n_0$ vertices has an injective harmonious labelling with at most $(1+\eps)n$ labels.
\end{proof}

\noindent
\textbf{Orthogonal double covers.}
Theorem~\ref{almostspan} can also be used to obtain an asymptotic solution for another old graph decomposition problem. An {\em orthogonal double cover} of the complete graph $K_n$ by a graph $G$ is a collection $G_1,\ldots,G_n$ of subgraphs of $K_n$ such that each $G_i$ is a copy of $G$, every edge of $K_n$ belongs to exacly two of the copies and any two copies have exactly one edge in common. The study of orthogonal double covers was originally motivated by problems in statistical design theory (see \cite{dinitz1992contemporary}, Chapter 2).
Since the number of copies in the double cover is $n$, it follows that $G$ must have $n-1$ edges.
The central problem here is to determine for which graphs $G$ there is a an orthogonal double cover of $K_n$ by $G$. In this full generality, this extends the existence question for both biplanes and
symmetric $2$-designs (see \cite{hughes1978biplanes}), and so must be considered difficult. About 20 years ago, Gronau, Mullin, and Rosa \cite{gronau1997orthogonal} made the following conjecture about trees.
\begin{conjecture}[Gronau, Mullin, and Rosa] \label{Conjecture_Orthogonal}
Let $T$ be an $n$-vertex tree which is not a path on $3$ edges. Then, $K_n$ has an orthogonal double cover by copies of $T$.
\end{conjecture}
This conjecture is known to hold for certain classes of trees including stars, trees with diameter $\leq 3$, comets, and trees with $\leq 13$ vertices (see \cite{gronau1997orthogonal, leck1997orthogonal}).
To see the connection between Conjecture~\ref{Conjecture_Orthogonal} and Theorem~\ref{almostspan}, we will consider a colouring of a complete graph on $2^k$ vertices, where edges are coloured by the sum of their endpoints in the abelian group $\mathbb{Z}_2^k$. By considering such a colouring we can show that Conjecture~\ref{Conjecture_Orthogonal} is asymptotically true whenever $n$ is a power of $2$.
\begin{theorem}\label{orthogonal}
Let $n=2^k$ and let $T$ be a tree on $n-o(n)$ vertices. Then $K_n$ contains $n$ copies of $T$ such that
every edge of $K_n$ belongs to at most two copies and any two copies have at most one edge in common.
\end{theorem}
\begin{proof}
Identify $V(K_n)$ with the group $\mathbb{Z}_2^k$. Colour each edge $ij$ with $i+j\in \mathbb{Z}_2^k$. By Theorem~\ref{almostspan}, $K_n$ has a rainbow copy $S$ of $T$.  For all $x\in \mathbb{Z}_2^k$, define a permutation $\phi_x:V(K_n)\to V(K_n)$ by $\phi_x(v)=x+v$ (with addition in $\mathbb{Z}_2^k$).
Use $\phi_x(S)$ to denote the subgraph of $K_n$ with edges $\{\phi_x(a)\phi_x(b): ab\in E(S)\}$. Notice that, since $\phi_x$ is a permutation of $V(K_n)$, $\phi_x(S)$ is a tree isomorphic to $T$. We claim that the family of $n$ trees $\{\phi_x(S): x\in \mathbb{Z}_2^k\}$ satisfies the theorem.

Notice that since $\phi_x(a)+\phi_x(b)=a+x+b+x=a+b$, the permutations $\phi_x$ preserve the colours of edges.
This implies that the trees $T_x$ are all rainbow.
Notice that the only edges fixed by the permutations $\phi_x$ are those coloured by $x$.
Finally, notice that $\phi_{x-y}\circ \phi_y=\phi_{x}$. Combining these, we have that if a colour $c$ edge is in two trees $\phi_x(S)$ and $\phi_y(S)$, then $x-y=c$ in $\mathbb{Z}_2^k$. This implies that the trees $\phi_x(S)$ cover any edge at most twice, and any pair of them have at most one edge in common.
\end{proof}

\medskip

The rest of this paper is organized as follows. In the next section we sketch the proof of Theorem~\ref{almostspan}. In Section~\ref{SectionPreliminaries} we define our notation, and recall some probabilistic results. In Sections~\ref{treesplit} -- \ref{Section_Almost_Spanning_Trees} we prove Theorem~\ref{almostspan}.

\section{Sketch of the proof of Theorem~\ref{almostspan}}\label{proofsketch}
In this section we sketch a proof of Theorem~\ref{almostspan}. For simplicity assume that $K_n$ is properly coloured, i.e.\ that $k=1$.
The proof splits the tree into a sequence of subforests, and iteratively extends the embedding to each subforest. The splitting of the tree goes as follows. We show that every tree has a short sequence $T=T_{\ell}\supseteq T_{\ell-1}\supseteq\dots \supseteq T_1\supseteq T_0$ where $|T_0|= o(n)$, and each forest $T_i$ is constructed from the previous one using one of the following three operations:
\begin{enumerate}[label = (\arabic{enumi})]
\item \label{AddStars} Add large stars  whose centers are in $T_{i-1}$.
\item \label{AddPaths} Add paths of length $3$ whose endvertices are in $T_{i-1}$.
\item \label{AddMatchings} Add a large matching one side of which is in $T_{i-1}$.
\end{enumerate}
There are two additional properties we can ensure:  \ref{AddStars} only needs to be performed once, when going from $T_0$ to $T_1$, and in  \ref{AddPaths}, the total number of vertices contained in all the paths that need to be added is $o(n)$.  See Lemma~\ref{treesplit} for a precise statement of this splitting.

To find a rainbow embedding of $T$, we start with a rainbow embedding of $T_0$ and iteratively extend it to rainbow embeddings of $T_1, \dots, T_{\ell}$ by performing one of the operations \ref{AddStars} -- \ref{AddMatchings}.
We use different proof techniques for performing \ref{AddStars} -- \ref{AddMatchings}:  \ref{AddStars} is done deterministically, whereas \ref{AddPaths} and \ref{AddMatchings} are done probabilistically. This interplay between deterministic and probabilistic techniques is one of the main new ideas introduced in this paper.

\smallskip

\textbf{Stars:} To find large stars in \ref{AddStars}, we use the deterministic technique of ``switchings''. The particular technique that we use originated in the papers of Woolbright \cite{woolbright78} and Brouwer, de Vries, Wieringa~\cite{ BVW78} about transversals in Latin squares.
The idea is to consider a maximal rainbow family of stars, and show that it must be large enough for our purposes. If the family of stars is not large enough, then it is possible to perform local manipulations to turn it into a larger rainbow family of stars (contradicting the maximality). See Section~\ref{sec:starfind} for further details.

\smallskip

\textbf{Paths and matchings:} To find paths and matchings in \ref{AddPaths} and \ref{AddMatchings}, we use the probabilistic method. That is, we choose a random set of vertices $X$ and a random set of colours $C$, and then try to find the paths/matching using only vertices in $X$ and  colours in $C$.  Here ``random set'' means that we choose every vertex/colour in $K_n$ independently at random with some fixed probability $p$.

To show that it is possible to find rainbow paths/matchings using random sets $X$ and $C$, we show that with high probability the subgraph of $K_n$ on $X$ and $C$ has certain pseudorandom properties. There are two relevant pseudorandom properties.
Firstly, we show that the subgraph $G$ of $K_n$  of edges with colours in $C$ has roughly the same edge-distribution as an Erd\H{o}s-Renyi random graph. Specifically, every vertex has degree about $pn$ in $G$ and every pair of disjoint vertex sets $A$, $B$ have about $p|A||B|$ edges between them.
Such a result was first proved by Alon, Pokrovskiy, and Sudakov in \cite{alon2016random}. We use a generalization of their result to locally bounded colourings (see Lemma~\ref{boundrandcolour}).
Secondly, we show that for random sets $X\subseteq V(K_n)$ and $C\subseteq C(K_n)$, the number of colours of $C$ between $X$ and any set $A\subseteq V(K_n)$
 is at least $(1-o(1))|A|$. See Lemma~\ref{randcol} for the precise statement.

Once we have established these pseudorandom properties, we use them to embed the paths and matchings for  \ref{AddPaths} and \ref{AddMatchings}. Embedding paths is easy --- since we only look for $o(n)$ of them, there is enough room to find them greedily (see Lemma~\ref{pathfinder}). Embedding matchings is harder, since there is less extra room. Matchings are embedded using a switching argument (like that used for stars), but one which exploits our pseudorandom properties (see Lemma~\ref{LemmaRainbowMatching}).

\smallskip

\textbf{Combining:} There is one difficulty left --- how do we combine the deterministic arguments for stars with the random ones for paths/matchings? The issue here is that when we deterministically embed the stars, we may not have control over which colours and vertices we use.  The vertices and colours used for \ref{AddPaths} and \ref{AddMatchings} need to be disjoint from those used for the stars in \ref{AddStars} and need to be random subsets of $V(K_n)$ and $C(K_n)$ respectively. These two requirements are incompatible with  our aim to choose the vertices and colours in \ref{AddStars} deterministically.

We get around this issue by randomizing the stars we build in \ref{AddStars}. Specifically we deterministically embed stars which are bigger than we need, and then randomly delete each vertex with fixed probability. The result is that we find the stars we want, with sufficient randomness in the unused vertices and colours. There is a  complication that arises with this argument --- there will be a dependency between the colours used on the stars and their vertices. Because of this the setting of several of the lemmas in this paper is the following: in a properly coloured $K_n$, we randomly choose a sets $X$ and $C$ of vertices and colours. The vertices are chosen independently of each other with probability $p$. The colours are chosen independently of each other with probability $p$. However, there may be arbitrary dependencies between the vertices and the colours. It turns out that the methods we use for \ref{AddMatchings} still work with these added dependencies, allowing us to combine \ref{AddStars} --  \ref{AddMatchings} to embed the whole tree $T$. More detail is given on this in Section~\ref{probexplainer}.

\section{Preliminaries}\label{SectionPreliminaries}
In this section, we first give some definitions and some notation, and then recall the concentration inequalities that we will use.
\subsection{Definitions and notation}
We use standard graph theory notation as well as the following definitions and notation relevant to a chosen colour class $C$ of $G$.
We say a graph is \emph{$\CC$-rainbow} if each of its edges has a different colour in $\CC$. A \emph{$C$-edge} is one with colour in $C$.
We say a vertex $x$ is a \emph{colour-$\CC$ neighbour} of $v$ in a graph $G$ if $x\in N_G(v)$, the neighbourhood of $x$ in $G$, and $xv$ has colour in $\CC$. We denote the set of colour-$C$ neighbours of $v$ in $G$ by $N_C(v)$, and define $N_C(A)=(\cup_{x\in A}N_C(x))\setminus A$ for each $A\subset V(G)$.
For a colour $c$, we make $N_c(x)=N_{\{c\}}(x)$ and other similar abbreviations.
Let $C(G)$ be the set of the colours of the edges of $G$. For an edge $e\in E(G)$, let $c(e)$ be the colour of $e$.

We also use the following two important definitions when splitting trees.
\begin{defn}\label{defbare}
A \emph{bare path} $P$ in a tree $T$ is a path whose interior vertices all have degree 2 in $T$. Where $P$ is a bare path in $T$, $T-P$ is the graph $T$ with the edges of $P$ and the interior vertices of $P$ deleted.
\end{defn}

\begin{defn}\label{defnnon}
In a tree $T$, we say $L$ is a set of \emph{non-neighbouring leaves} if $L$ is a set of leaves which pairwise share no neighbours. In other words, removing $L$ from $T$ removes a matching.
\end{defn}

We use common asymptotic notation for any strictly positive real functions $f$ and $g$, as follows. If $f(n)/g(n)\to 0$ as $n\to\infty$, then we say $f=o(g)$ and $g=\omega(f)$. If there is some constant $C$ for which $|f(n)|\leq C|g(n)|$ for all $n$, then we say $f=O(g)$ and $g=\Omega(f)$.
In addition, we use the \emph{hierachy} $x \gg y\gg z$ if there exists some non-zero decreasing functions $f$ and $g$ such that if $y\leq f(x)$ and $z\leq g(y)$ then all the subsequent inequalities we need concerning $x$, $y$ and $z$ hold. Where we say a property almost surely holds or holds with probability $1-o(1)$ in conjunction with a hierachy $x\gg y\gg 1/n$, we mean that, for each $\eps>0$, and $x$ and $y$ with $x\ll y$, for all sufficiently large $n$ the property holds with probability at least $1-\eps$. In all our lemmas we assume that $n$ is an integer which is sufficiently large (e.g.\ bigger than $10^{6}$ is sufficient).

\subsection{Concentration inequalities}
We make use of Chernoff's lemma in the following form (see, for example,~\cite{janson2011random}).
\begin{lemma}\label{chernoff} If $X$ is a binomial variable with standard parameters~$n$ and $p$, denoted $X=\mathrm{Bin}(n,p)$, and $\e$ satisfies $0<\e\leq 3/2$, then
\[
\P(|X-\E X|\geq \e \E X)\leq 2\exp\left(-\e^2\E X/3\right).\hfill\qedhere
\]
\end{lemma}

Often, the random variables we consider will depend on both random vertices and colours, with some dependencies between the vertices and colours. Here, we will use Azuma's inequality, for which we need the following definition. Given a probability space $\Omega=\prod_{i=1}^n\Omega_i$ and a random variable $X:\Omega\to \mathbb{R}$ we make the following definition. If there is a constant $k$ such that changing $\omega\in \Omega$ in any one coordinate changes $X(\omega)$ by at most $k$, then we say that $X$ is $k$-Lipschitz.

\begin{lemma}[Azuma's Inequality]\label{Lemma_Azuma}
Suppose that $X:\prod_{i=1}^n\Omega_i\to \mathbb{R}$ is $k$-Lipschitz. Then
$$\P\left(|X-\E X|>t \right)\leq 2\exp(-t^2/k^2n).$$
\end{lemma}


\section{Tree splitting}\label{treesplit}
We will split each tree we seek to embed by initially removing large matchings iteratively (that is, removing non-neighbouring leaves, see Definition~\ref{defnnon}). When this is no longer possible, either few vertices remain, or we  have many disjoint long bare paths (see Definition~\ref{defbare}), as shown by Lemma~\ref{findpaths}. Where these long paths occur, we will remove a path with length 3 from each end of each path, before iteratively removing matchings to remove the remaining edges in the middle of the paths. Few vertices will then remain, giving us the required splitting, in Lemma~\ref{decomp}.

\begin{lemma}\label{findpaths} Let $\ell,m\geq 2$. Suppose $T$ is a tree with at most $\ell$ leaves. Then, there is some $s$ and some vertex-disjoint bare paths $P_i$, $i\in [s]$, in $T$ with length $m$ so that $|T-P_1-\ldots-P_s|\leq 6m\ell+2|T|/(m+1)$.
\end{lemma}
\begin{proof} For the appropriate $r$, let $Q_1,\ldots,Q_r$ be the maximal bare paths in $T$. Note that each edge is in some path $Q_i$, and that these paths are edge-disjoint. Replacing each path $Q_i$ by an edge gives a tree, $S$ say, with $r$ edges, no vertices of degree 2, and at most $\ell$ leaves.
Since $S$ has $r+1-\ell$ vertices of degree $\geq 3$, the sum of the degrees gives $\ell+  3(r+1-\ell)\leq \sum_{v\in V(S)}d_S(v)= 2r$, which implies $r\leq 2\ell$.

For each $i\in[r]$, find within $Q_i$ as many vertex-disjoint length $r$ subpaths as possible while avoiding the endvertices of $Q_i$ (thus the subpaths from different $Q_i$ are vertex-disjoint), say the paths  $Q_{i,j}$, with $j\in [j_i]$, where due to the maximality $j_i=\lfloor(|Q_i|-2)/(m+1)\rfloor\geq (|Q_i|-m-2)/(m+1)$. Removing the internal vertices of a subpath with length $m$ from $Q_i$ removes $m-1$ vertices. Thus, for each $i\in [r]$, we have
\begin{equation}\label{moredetail}
|Q_i-\cup_{j\in [j_i]}Q_{i,j}|= |Q_i|-(m-1)j_i\leq \frac{(m+1)|Q_i|-(m-1)(|Q_i|-m-2)}{m+1}\leq 2m+\frac{2|Q_i|}{m+1}.
\end{equation}
Note that
\begin{align*}
\Big|T-{\bigcup_{i\in [r]}}{\bigcup_{j\in [j_i]}}Q_{i,j}\Big|&\leq \sum_{i\in [r]}\Big|Q_i-{\bigcup_{j\in [j_i]}}Q_{i,j}\Big|
\leq \sum_{i\in [r]}(2m+2|Q_i|/(m+1))
\\
&\leq 2mr+2(|T|+2r)/(m+1)\leq 6m\ell+2|T|/(m+1).
\end{align*}
The first inequality comes from $T=\cup_{i\in [r]}Q_i$, while the second comes from~\eqref{moredetail}. The third inequality comes from $e(T)=\sum_{i\in [r]}(|Q_i|-1)= |T|-1$, and the last inequality comes from $r\leq 2\ell$.
Thus, the set of paths $\{Q_{i,j}:i\in[r],j\in[j_i]\}$ have the property required.
\end{proof}

Using Lemma~\ref{findpaths}, we can now find the desired splitting of an arbitrary tree.

 \begin{lemma}\label{decomp} Given integers $D$ and $n$, $\mu>0$ and a tree $T$ with at most $n$ vertices, there are integers $\ell\leq 10^4 D\mu^{-2}$ and $j\in\{2,\ldots,\ell\}$ and a sequence of subgraphs $T_0\subset T_1\subset \ldots \subset T_\ell=T$ such that
\begin{enumerate}[label = \textbf{P\arabic{enumi}}]
\item for each $i\in [\ell]\setminus \{1,j\}$, $T_{i}$ is formed from $T_{i-1}$ by adding non-neighbouring leaves,\label{cond1}
\item $T_j$ is formed from $T_{j-1}$ by adding at most $\mu n$ vertex-disjoint bare paths with length $3$,\label{cond2}
\item $T_1$ is formed from $T_0$ by adding vertex-disjoint stars with at least $D$ leaves each, and\label{cond3}
\item $|T_0|\leq \mu n$.\label{cond4}
\end{enumerate}
\end{lemma}

\begin{proof} Let $m=\lceil 100/\mu\rceil$.
Iteratively remove sets of $\mu n/16 m D$ non-neighbouring leaves from $T$ as many times as possible. Eventually we obtain a subtree which has $< \mu n/16 m D$ non-neighbouring leaves.  Call this subtree $T_{m-4}$, and let  $T_{m-3}\subset \ldots \subset T_\ell=T$ be the intermediate subtrees, i.e.\ where $T_{i}$ is formed from $T_{i-1}$ by adding  $\mu n/16 m D$ non-neighbouring leaves. Note that $\ell\leq m+16mD/\mu\leq 10^4 D\mu^{-2}$.
Fix a maximal collection of vertex-disjoint bare paths with length $m$ in $T_{m-4}$.
Remove the subpath with length 3 from each end of each path from $T_{m-4}$, and call the resulting forest $T_{m-5}$ {(this is possible because $m\geq 8$)}.
Note that we removed at most ${2}n/m\leq \mu n$ such paths. For each $i=m-5,m-6,\ldots,2$ in turn, remove a leaf from each disconnected bare path of length $i-1$ in $T_{i}$ to get the forest $T_{i-1}$.
Note that, to get the forest $T_1$, we have removed every internal vertex except for one from each path in the maximal collection of vertex-disjoint bare paths with length $m$.

{Finally, for every vertex $v\in T_1$ which has $\geq D$ leaves in $T_1$,  remove all of the leaves of $v$. Call the resulting forest $T_0$.}
 We will show that $|T_0|\leq \mu n$, whereupon the sequence $T_0\subset\ldots \subset T_\ell=T$ satisfies \ref{cond1}
to \ref{cond4} with $j=m-4$.

Consider $S=T_{m-4}$ and remove any leaves around a vertex if it has at least $D$ leaves. Call the resulting tree $S'$. By the choice of $\ell$, at most $\mu n/16mD$ vertices in $S$ have some incident leaf. Each leaf in $S'$ had an incident leaf in $S$ (which was then removed) or was incident to a vertex in $S$ which had at most $D-1$ leaves. Thus, $S'$ has at most $\mu n/16m$ leaves.

By Lemma~\ref{findpaths},  there are vertex-disjoint bare paths $P_1,\ldots,P_r$, for some $r$, with length $m$,
so that $|S'-P_1-\ldots-P_r|\leq 6\mu n /16+2n/m \leq \mu n/2$. As at most $\mu n/16mD$ vertices in $S$ have incident leaves,
all but at most $\mu n/16mD$ of these paths $P_i$ are bare paths in $S$.
At least $r-n/16mD$ many vertex-disjoint bare paths with length $m$ must have
 been removed from $S=T_{m-4}$, except for one vertex in each, to get $T_1$,
so therefore $T_1$ can have at most $\mu n/2+r+m\cdot \mu n/16mD\leq \mu n$
vertices which are not in $V(S)\setminus V(S')$ {(using $r\leq n/m\leq \mu n/100$)}.
As every vertex in $V(S)\setminus V(S')$ is removed from $T_1$ in creating $T_0$, we thus have $|T_0|\leq \mu n$, as required.
\end{proof}


\section{Choosing random colours}\label{sec:randcolour}
Consider a  properly coloured complete graph $K_n$, and form a random set of colours by choosing every colour independently with probability $p$.
Let $G$ be the subgraph of $K_n$ consisting of the edges of the chosen colours.
Alon, Pokrovskiy and Sudakov~\cite{alon2016random} showed that, with high probability, the subgraph $G$ has  a similar edge distribution to the binomial random graph. Specifically they showed that, for any two sets of vertices $A$ and $B$, the number of edges in the chosen colours between $A$ and $B$ is concentrated around the expectation $p|A||B|$. In this section we will prove a generalization of this result on proper colourings to colourings of $K_n$ which are locally $k$-bounded.
The method we use is similar to the proof from~\cite{alon2016random}.

\begin{lemma}\label{boundrandcolour} Let $\e>0$ and $k\in\N$ be constant and let $p\geq n^{-1/100}$. Let $K_n$ have a {locally}  $k$-bounded colouring and suppose $G$ is a subgraph of $K_n$ chosen by including the edges of each colour independently at random with probability $p$. Then, with probability $1-o(n^{-1})$, for any disjoint sets $A,B\subset V(G)$, with $|A|,|B|\geq n^{3/4}$,
\[
\big|e_G(A,B)-p|A||B|\big|\leq \e p|A||B|.
\]
\end{lemma}
\begin{proof}
We will first show that, with high probability, the property we want holds for pairs of sets with many colours between them.
\begin{claim}\label{claimnew1}
With probability $1-o(n^{-1})$ the following holds.
\begin{enumerate}[label = \textbf P]
\item For any disjoint sets $A,B\subset V(K_n)$, with $|A| \geq |B|\geq n^{1/10}$, which have at least $(1-\e p/8)|A||B|$ different colours between $A$ and $B$ in $K_n$, we have $|e_G(A,B)-p|A||B||\leq \e p|A||B|/2$. \label{propP}
\end{enumerate}
\end{claim}
\begin{proof}
For any such sets $A$ and $B$, we can select a rainbow subgraph $R$ of $K_n[A,B]$ with $(1-\e p/8)|A||B|$ edges. {Notice that $e(R\cap G)\sim \mathrm{Bin}((1-\e p/8)|A||B|, p)$.} By Lemma~\ref{chernoff} {applied with $\varepsilon p/8$ for $\eps$}, with probability at least $1-\exp(-\e^2p^3|A||B|/10^3)$, we have
 $(1-\e p/4)p|A||B|\leq e(R\cap G)\leq (1+\e p/4)p|A||B|$, which in combination with $e(K_n[A,B]-R)\leq \e p|A||B|/8$ implies that \ref{propP} holds for $A$ and $B$.
Note that, for such sets, we have $p^3|B|\geq n^{7/100}=\omega(\log n)$, so when $|A|\geq |B|\geq n^{1/10}$, condition \ref{propP} holds for $A$ and $B$ with probability at least $1-\exp(-|A|\cdot\omega(\log n))$.

Thus, \ref{propP} holds with probability at least
\[
1-\sum_{b=n^{1/10}}^n\sum_{a=b}^n\binom{n}{a}\binom{n}{b}\exp(-a\cdot\omega(\log n))=1-o(n^{-1}).\hfill\qedhere
\]
\end{proof}

Assuming that~\ref{propP} holds, we now show that the property in the lemma holds.
Let $A,B\subset V(K_n)$ be disjoint sets with $|A|\geq|B|\geq n^{3/4}$.
 Let $\ell= \lceil 10\sqrt{kn}/\e^2p^2\rceil \leq n^{0.6+o(1)}$.

{\begin{claim}\label{claimnew2}
There are partitions $A=A_1\cup\ldots\cup A_\ell$ and $B=B_1\cup\ldots\cup B_\ell$ such there are at most $\e^2 p^2 |A||B|/100$ edges $ab$ between $A$ and $B$ for which there is another edge $a'b'$ with $c(ab)=c(a'b')$ and $ab, a'b'\in E(K_n[A_i, B_j])$ for some $i,j\in[\ell]$.
\end{claim} }
 \begin{proof}
 Pick random partitions $A=A_1\cup\ldots\cup A_\ell$ and $B=B_1\cup\ldots\cup B_\ell$, by choosing the part of each element independently and uniformly at random. Fix any colour-$c$ edge $e$, and let  $A_i$, $B_j$ be the classes it goes between.
The probability that there is another colour-$c$ edge between $A_i$ and $B_j$ sharing a vertex with $e$ is $\leq 2(k-1)/\ell$ (since $e$ touches at most $2(k-1)$ other colour-$c$ edges).
The probability there is a colour-$c$ edge disjoint from $e$ between $A_i$ and $B_j$ is $\leq kn/2\ell^2$ (since that are at most $kn/2$ colour-$c$ edges in total).
Combining these, the probability there is another colour-$c$ edge between $A_i$ and $B_j$ is at most
\[
\frac{2(k-1)}{\ell}+\frac{kn}{2\ell^2}\leq \frac{kn}{\ell^2}\leq\frac{\e^2p^2}{100}.
\]
Thus, the expected number of edges which have a non-unique colour across their classes is at most $\e^2p^2 |A||B|/100$. Fix such a partition then for which there are at most $\e^2 p^2 |A||B|/100$ such edges.
 \end{proof}

{Fix a partition $A=A_1\cup\ldots\cup A_\ell$ and $B=B_1\cup\ldots\cup B_\ell$ from the above claim.}
{\begin{claim}\label{claimnew3}
Let $H$ be the subgraph of $K_n[A,B]$ with all the edges between any pair $A_i$, $B_j$ which does not satisfy \ref{propP} removed. Then
\begin{equation}\label{edgesdeleted}
|e_H(A,B)-|A||B||\leq \e p |A||B|/4.
\end{equation}
\end{claim}}
\begin{proof}
First, delete any edge adjacent to a class $A_i$ or $B_j$ which contains at most $n^{1/10}$ vertices. For large $n$, the number of edges deleted in this first stage is at most
\begin{equation}\label{iseefire}
\ell \cdot n^{1/10}(|A|+|B|)\leq n^{0.7+o(1)}(|A|+|B|)\leq\e p|A||B|/8.
\end{equation}
Secondly, delete edges between any class $A_i$ and $B_j$ if there are at most $(1-\e p/8)|A_i||B_j|$ different colours between $A_i$ and $B_j$ in $K_n$. Note that, at each such deletion, we lose
$|A_i||B_j|$ edges, at least $\e p|A_i||B_j|/8$ of which have non-unique colours between their classes.  Thus, if $I$ is the set of pairs $(i,j)$ for which we deleted edges between $A_i$ and $B_j$ for this reason, we have
\[
\sum_{(i,j)\in I}\e p |A_i||B_j|/8\leq \e^2p^2|A||B|/100.
\]
Thus, we have deleted at most $\sum_{(i,j)\in I}|A_i||B_j|\leq \e p|A||B|/8$ edges in this second stage, and, by~\eqref{iseefire}, at most $\e p|A||B|/4$ edges in total, proving (\ref{edgesdeleted}).
\end{proof}

Let $I'$ be the set of pairs $(i,j)$ for which there remain edges between $A_i$ and $B_j$ in $H$. By~\ref{propP}, for each $(i,j)\in I'$, we have
\begin{align}
|e_{G\cap H}(A,B)-p\cdot e_H(A,B)|&={\Big|\sum_{(i,j)\in I'} e_{G}(A_i,B_j)-p|A_i||B_j|\Big|}
\leq \sum_{(i,j)\in I'}|e_G(A_i,B_j)-p|A_i||B_j|| \notag\\
&\leq \sum_{(i,j)\in I'}\e p|A_i||B_j|/2\leq \e p|A||B|/2. \label{Eq_Edges_H}
\end{align}
{The first inequality is the triangle inequality. The second inequality is~\ref{propP}. The third inequality follows as $A_1\cup\ldots\cup A_\ell$ and $B_1\cup\ldots\cup B_\ell$ partition $A$ and $B$ respectively.}
The required property then holds for $A$ and $B$, as follows.
\begin{align*}
|e_G(A,B)-p|A||B||
&\leq |e_G(A,B)-e_{G\cap H}(A,B)|+|e_{G\cap H}(A,B)-p\cdot e_H(A,B)|+p|e_H(A,B)-|A||B||\\
&\leq 2|e_H(A,B)-|A||B||+\e p |A||B|/2 \leq\e p |A||B|.
\end{align*}
{The first inequality is the triangle inequality. The third inequality comes from (\ref{edgesdeleted}).
The second inequality comes from (\ref{Eq_Edges_H}) and as $E_G(A,B)\setminus E_{G\cap H}(A,B)\subseteq E_{K_{n}}(A,B)\setminus E_{H}(A,B)$.}
\end{proof}


\section{Colours and random vertex sets}\label{sec:randvertex}
Consider a locally $k$-bounded coloured complete graph $K_n$, and let $C$ be a random set of colours chosen independently with probability $p$.
Consider two sets of vertices $A$ and $B$. In the previous section, we showed that the number of colour-$C$ edges between $A$ and $X$ is about $p|A||B|$. Since
every colour can occur $\leq k|A|$ times between $A$ and $B$, this implies that the number of colours of $C$ occuring between $A$ and $B$ is at least $(1-o(1))p|B|/k$.
Our goal here is to prove that if $B$ is  a \emph{random set of of vertices}, then this number increases to $(1-o(1))|B|/k$. In fact, the following main result of this section shows
that this holds for every large subset $B$ of a random set of vertices $X$.

\begin{lemma}\label{randcol} Let $k\in\N$ and $\e>0$ be constant, and  $p\geq n^{-1/10^4}$. Let $K_n$ have a {locally}  $k$-bounded colouring.  Let $X\subset V(K_n)$ and $C\subset C(K_n)$ be random subsets where, for each $x\in V(K_n)$ and $c\in C(K_n)$, $\P(x\in X)=\P(c\in C)=p$, all the events $\{x\in X\}$ are independent and all the events
 $\{c\in C\}$ are independent (but the event $\{x\in X\}$ might depend on the events $\{c\in C\}$).
Then, with probability $1-o(n^{-1})$, for each $A\subset V(K_n)\setminus X$ and $B\subset X$ with $|A|\geq n^{3/4}$ and $|B|\geq \e pn$,
there are at least $(1-\e)|B|/k$ colours which appear in $G$ between $A$ and $B$.
\end{lemma}

To prove this lemma we consider a set $A$ and a random set $X$, and analyse how many times each colour can appear between them. Without any assumptions on $X$, we know that every colour  appears $\leq k|A|$ times between $A$ and $X$ (from local $k$-boundedness). The following lemma shows that if $X$ is chosen randomly with probability $p$, then \emph{most} colours appear only at most $(1+o(1))pk|A|$ times between $A$ and $X$.

\begin{lemma}\label{randset} Let $k$ be constant and  $\e,p\geq n^{-1/100}$. Let $K_n$ have a {locally}  $k$-bounded colouring and let $X$ be a random subset of $V(K_n)$ with each vertex included independently at random with probability $p$. Then, with probability $1-o(n^{-1})$, for each $A\subset V(K_n)$ with $n^{1/20}\leq |A|\leq n^{1/4}$,
for all but at most $\eps n$ colours there are at most $(1+\e)pk|A|$ edges of that colour between $A$ and $X$.
\end{lemma}

\begin{proof}
Let $\ell=\lceil kn^{1/2}\rceil$. Fix $A\subset V(G)$ with $n^{1/20}\leq |A|\leq n^{1/4}$. For each $c\in \CC(K_n)$, let $A_c=N_c(A)$, so that $|A_c|\leq k|A|$. Let $C'({K_n})$ be the colours with more than $pk|A|$ edges between $A$ and $V(K_n)$ in $K_n$, so that $|C'({K_n})|\leq n/pk$.
\begin{claim}\label{claim1} There is a partition $\CC'(K_n)=\CC_1\cup\ldots\cup\CC_\ell$ so that, for each $i\in[\ell]$ and $a,b\in \CC_i$,
$A_a$ and $A_b$ are disjoint.
\end{claim}
\begin{proof} Create an auxillary graph with vertex set $C'({K_n})$ where $ab$ is an edge exactly if $A_a\cap A_b\neq\emptyset$. For any colour $c\in C({K_n})$, there are at most $k|A|^2\leq \ell$ edges between $A$ and $A_c$, and hence at most $\ell-1$ different colours from $C'({K_n})\setminus \{c\}$ on such edges. The auxillary graph then must have maximum degree at most $\ell-1$, and is thus $\ell$-colourable, so the claim holds.
\end{proof}

Now, consider $X$, the random subset of $V(K_n)$ with each vertex included independently at random with probability $p$. For each $i\in[\ell]$, let
\[
\XX_i=\{c\in \CC_i:|A_c\cap X|> (1+\e)pk|A|\}
\]
and $\XX=\cup_{i\in [\ell]}\XX_i$. To show that $|B| \leq \e n$, we will use the following claim.

\begin{claim}\label{claim2} For each $i\in [\ell]$, if $|\CC_i|\geq \eps n/2\ell$, then $\P(|\XX_i|\geq \eps p|\CC_i|/2)\leq\exp(-\e^2pn^{1/2}/{200k})$.
\end{claim}
\begin{proof} Recall that $|A|\geq n^{1/20}$ and $\e,p\geq n^{-1/100}$, so that $\eps^2 pk|A|=\omega(-\log(\e p))$. {Note that $|A_c\cap X|\sim \mathrm{Bin}(|A_c|, p)$.}
By Lemma~\ref{chernoff}, {and as $|A_c|\leq k|A|$}, for each $c\in \CC_i$,
\[
\P(c\in \XX_i)\leq 2\exp(-\eps^2 pk|A|/3)\leq \eps p/4,
\]
for sufficiently large $n$. Furthermore, the events $\{c\in \XX_i\}$, $c\in \CC_i$, are independent (by the disjointness of the sets $A_c$ for $c\in C_i$). This implies that $|B_i|$ is stochastically dominated by $\mathrm{Bin}(|\CC_i|, \varepsilon p/4)$.
By Lemma~\ref{chernoff} {applied with $1/2$ for $\varepsilon$},
if $|\CC_i|\geq \eps n/2\ell$, then
\[
\P(|\XX_i|\geq \eps p|\CC_i|/2)\leq 2\exp(- (\eps p|\CC_i|/4)/12)\leq \exp(-\e^2 pn/100\ell)\leq\exp(-\e^2 pn^{1/2}/{200k}).\qedhere
\]
\end{proof}

Note that from our choice of parameters we have that $\e^2 pn^{1/2}/{200k}>n^{2/5}$. Therefore the probability that
for some subset $A\subset V(K_n)$, with $n^{1/20}\leq |A|\leq n^{1/4}$, there is an index $i\in [\ell]$ such that
$|\CC_i|\geq \e n/2\ell$ and $|\XX_i|\geq \e p|C_i|/2$ is at most
 \[
 \sum_{a=n^{1/20}}^{n^{1/4}}\binom{n}{a}\cdot \ell \cdot \exp(-n^{2/5})=o(n^{-1}).
 \]
Therefore, with high probability, we can assume that for every subset $A$ and every $C_i$ with  $|\CC_i|\geq \e n/2\ell$  the corresponding $\XX_i$ satisfies $|\XX_i|\leq \e p|C_i|/2$.
Then,
\begin{eqnarray*}
|B| &=& \sum_{i\in [\ell]} |\XX_i| \leq \sum_{i:|\CC_i|\geq \e n/2\ell}|\XX_i|+\sum_{i:|\CC_i|< \e n/2\ell}|\CC_i| \\
&\leq& \frac{\e p}{2} \sum_{i\in [\ell]} |\CC_i|+ \ell \cdot \frac{\e n}{2\ell} \leq \frac{\e n}{2k} + \frac{\e n}{2}<\e n,
\end{eqnarray*}
where we have used that $\sum_{i\in [\ell]}|C_i|{= |C'(K_n)|}\leq n/pk$.  \qedhere
\end{proof}

We can now show that the property in Lemma~\ref{randset} is likely also to hold for any subset $A$ which is not too small.

\begin{corollary}\label{randsetcor} Let $k$ be constant and  $\e,p\geq n^{-1/10^3}$. Let $K_n$ have a {locally}  $k$-bounded colouring and let $X$ be a  random subset of $V(K_n)$ with each vertex included independently at random with probability $p$. Then, with probability $1-o(n^{-1})$, for each $A\subset V(K_n)\setminus X$
 with $|A|\geq n^{1/4}$, for all but at most $\eps n$ colours there are at most $(1+\e)pk|A|$ edges of that colour between $A$ and $X$.
\end{corollary}
\begin{proof}
By Lemma~\ref{randset} {applied with $\varepsilon^2p/4$ for $\eps$}, with probability $1-o(n^{-1})$, for each subset $B\subset V(G)$ with $n^{1/20}\leq |B|\leq n^{1/4}$, for all but at most {$\eps^2 pn/4$} colours there are at most $(1+\e^2)pk|B|$ edges of that colour between $B$ and $X$.

Let $A\subset V(G)$ satisfy $|A|\geq n^{1/4}$ and choose $A=A_1\cup \ldots \cup A_\ell$ with $\ell=\lfloor |A|/n^{1/20} \rfloor$ and $n^{1/20}\leq |A_i|\leq 2 n^{1/20}$ for each $i\in[\ell]$. For each $i\in[\ell]$, let $C_i$ be the set of colours for which there are more than $(1+\e^2)pk|A_i|$
 edges of that colour between $A_i$ and $X$, so that $|C_i|\leq \eps^2 pn/4$.
  Let $C'$ be the set of colours for which there are more than $(1+\e)pk|A|$ edges of that colour between $A$ and $X$. We need then to show that $|C'|\leq \e n$.

Note that, if $c\in \CC'$, then
\begin{align}
(1+\e)pk|A|&\leq {\sum_{a\in A} |N_c(a)\cap X|\leq }
\sum_{i:c\in \CC_i}{k}|A_i|+ \sum_{i:c\notin \CC_i}(1+\e^2) pk|A_i|\nonumber \\
&\leq |\{i:c\in \CC_i\}|\cdot 2{k}n^{1/20}+(1+\e^2)pk|A|,\label{jeremy}
\end{align}
The first equality comes from $c\in C'$. The second inequality comes from the fact that the number of colour-$c$ edges between $A_i$ and $X$  is at most $(1+\e^2)pk|A_i|$ for $i\not\in \CC_i$, and at most $k|A_i|$ for all other $i$ (by the local $k$-boundedness of $K_n$). The third inequality comes from $|A_i|\leq 2 n^{1/20}$ and $\sum_{i\in [\ell]}|A_i|=|A|$.
This implies that
\[
\e pk|A|/2 \leq (\e-\e^2)pk|A|\overset{\eqref{jeremy}}{\leq} |\{i:c\in \CC_i\}|\cdot 2{k}n^{1/20}\leq |\{i:c\in \CC_i\}|\cdot 2k|A|/\ell.
\]
Thus, if $c\in \CC'$, then $|\{i:c\in \CC_i\}|\geq \e p\ell/4$. This gives
\[
|\CC'|\cdot \e p\ell/4\leq \sum_{c\in \CC'} |\{i:c\in \CC_i\}|=\sum_{i\in [\ell]}|\CC'\cap \CC_i|\leq \sum_{i\in [\ell]}|\CC_i|.
\]
 Therefore,
\[
|\CC'|\leq \frac{\sum_{i\in[\ell]}|\CC_i|}{\e p \ell/4}\leq \frac{\ell\cdot\e^2 pn/4}{\e p\ell/4}=\e n,
\]
as required.
\end{proof}

Combining the property in Corollary~\ref{randsetcor} with Lemma~\ref{boundrandcolour}, we can now complete the proof of the main result of this section.

\begin{proof} [Proof of Lemma~\ref{randcol} ]
{The desired property in the lemma strengthens as $\eps$ decreases, so we may assume that $\e\leq 1/2$.}
By Corollary~\ref{randsetcor} {applied with $\varepsilon^2p^2/4k$ for $\eps$}, with probability $1-o(n^{-1})$, for each $A\subset V(G)$ with $|A|\geq n^{3/4}$, for all but at most $\eps^2 p^2n/4k$ colours there are at most $(1+\e^2)pk|A|$ edges of that colour between $A$ and $X$ in $K_n$. With probability $1-o(n^{-1})$, by Lemma~\ref{boundrandcolour} {applied with $\varepsilon^2$ for $\eps$}, for every two disjoint subsets $A,B\subset V(G)$
 with $|A|,|B|\geq n^{3/4}$, we have $e_G(A,B)\geq (1-\e^2)p|A||B|$.

Now, for each $A\subset V(K_n)\setminus X$ and $B\subset X$  with $|A|\geq n^{3/4}$ and $|B|\geq \e pn\geq n^{3/4}$, there are at least $(1-\e^2)p|A||B|$ edges between $A$ and $B$ in $G$.
Delete all edges between $A$ and $B$ in $G$ {whose colour appears} more than $(1+\e^2)pk|A|$ times between $A$ and $B$. As each colour appears between $A$ and $B$ in $K_n$ at most $k|A|$ times, this removes at most $k|A|\cdot \eps^2 p^2n/4k$ edges.

Each remaining colour between $A$ and $B$ in $G$ occurs between $A$ and $B$ at most $(1+\e^2)pk|A|$ times in $G$. Therefore, between $A$ and $B$ in $G$ the  number of different remaining colours is at least
\[
\frac{(1-\e^2)p|A||B|-\eps^2 p^2n|A|/4}{(1+\e^2)pk|A|}\geq (1-\e/2)|B|/k-\e^2pn/2k\geq (1-\e)|B|/k,
\]
as required. {The first inequality uses $\e\leq 1/2$, while the last inequality uses $|B|\geq \varepsilon p n$.}
\end{proof}


\section{Finding a large matching}\label{sec:randmatching}

In this section we show that one can find an almost spanning rainbow matching between any set $A$ with appropriate size and a random set $X$.
Note that in the following lemma the random set of colours and the random set of vertices are not required to be independent of each other.
This is important for the application of this statement in our methods. More detail on this is given in Section~\ref{probexplainer}.

\begin{lemma} \label{LRainbowMatching}
Let $k\in \N$ and $\e>0$ be constant,  and suppose $K_n$ has a  {locally} $k$-bounded colouring and $p\geq n^{-1/10^4}$.  Let $X\subset V(K_n)$ and $C\subset C(K_n)$ be random subsets where, for each $x\in V(K_n)$ and $c\in C(K_n)$, $\P(x\in X)=\P(c\in C)=p$, and all the events $\{x\in X\}$ are independent and all the events $\{c\in C\}$ are independent.
Then, with probability $1-o(n^{-1})$, for each set $A\subset V(K_n)\setminus X$ with $|A|\leq pn/k$, there is a $C$-rainbow matching in $K_n$ of size at least $|A|-\e pn$
from $A$ into $X$.
\end{lemma}

To prove this lemma we first show, using fairly standard edge swapping arguments, how to find a rainbow matching in a bipartite graph which has edges of many different colours between
any two large subsets of vertices. Note that we established such a property in our setting in the previous section.

\begin{lemma} \label{LemmaRainbowMatching}
Let $\e\gg \eta \gg 1/n>0$.
 Let $G$ be a bipartite graph with classes $X$ and $Y$, with $|X|=n$ and $|Y|=kn$, which has a locally $k$-bounded colouring. Suppose that between any two subsets $A\subset X$ and $B\subset Y$ with size at least $\eta n$ there are at least $(1-\eta)|B|/k$ colours in $G$ which appear between $A$ and $B$. Then, there exists a rainbow matching in $G$ with at least $(1-\e)n$ edges.
\end{lemma}
\begin{proof}
Let $M$ be a maximal rainbow matching in $G$.  Let $A_0=X\setminus V(M)$ and $B_0=Y\setminus V(M)$. Suppose, for later contradiction, that $|A_0|\geq \e n$.

Letting $\ell=\lceil 4/\e\rceil$ and $s=5^{\ell+1}$, define a sequence of colour sets $C_0\subset C_1\subset \ldots\subset C_\ell$, and sequences of vertex sets $A_1\subset \ldots\subset A_\ell\subseteq X$ and $B_1\subset \ldots \subset B_\ell\subseteq Y$ recursively as follows, for each $i\geq 1$.
\begin{align*}
C_{i-1}&=\{\text{colours with at least $s$ disjoint edges between $A_{i-1}$ and $B_{i-1}$}\}\\
A_i&=A_0\cup\{u\in X: \exists v \in Y\text{ s.t.~}uv\text{ is an edge of $M$ with colour from $C_{i-1}$}\}\\
B_i&=B_0\cup\{v\in Y: \exists u \in X\text{ s.t.~}uv\text{ is an edge of $M$ with colour from $C_{i-1}$}\}
\end{align*}
Observe that, for each $0\leq i\leq \ell$, $M - (A_i\cup B_i)$ is a matching, and so $k|A_i|\leq |B_i|$. Let $C_{-1}=\emptyset$.

\begin{claim}\label{ClaimExtendMatching}
For each $0\leq i\leq \ell$ and any set of vertices $U\subseteq A_i \cup B_i$ with  $|U|\leq 5^{\ell+1-i}$, there is a rainbow matching $M_U$ in $G$, with $V(M_U)\cap U=\emptyset$ and $C(M_U)=C(M)$, which contains every edge in $M$ with colour outside $C_{i-1}$.
\end{claim}
\begin{proof}
The proof is by induction on $i$. For $i=0$, the claim holds, taking $M_U=M$ for any $U$. Suppose then that $i>0$ and the claim holds for $i-1$.

Let $U\subseteq A_i \cup B_i$ with $|U|\leq 5^{\ell+1-i}$. Let $U_1=U\cap (A_{i-1}\cup B_{i-1})$ and $U_2=U\setminus U_1$. By the definition of $A_i$ and $B_i$, there is a set of at most $|U_2|$ edges $M_2\subseteq M$ such that $U_2\subseteq V(M_2)$ and $C(M_2)\subseteq C_{i-1}\setminus C_{i-2}$.
Note that $V(M_2)\cap (A_{i-1}\cup B_{i-1})=\emptyset$.

By the definition of $C_{i-1}$, for every $m\in M_2$, there are at least $s$ disjoint edges between $A_{i-1}$ and $B_{i-1}$ with colour the same as $m$. Since $s= 5^{\ell+1}\geq 5\cdot 5^{\ell+1-i}\geq |U_1|+2|M_2|$, we can choose an edge, $e_m$ say, with colour the same as $m$, for all $m\in M_2$, such that
 $M_2'=\{e_m: m\in M_2\}$ is a matching, $V(M_2')\subset A_{i-1}\cup B_{i-1}$ and $V(M_2')\cap U_1=\emptyset$. Note that $|M_2'|=|M_2|$ and $C(M_2')=C(M_2)\subseteq C_{i-1}\setminus C_{i-2}$.

Let $U'=V(M_2')\cup U_1$, so that $U'\subseteq A_{i-1}\cup B_{i-1}$ and $|U'|\leq |U_1|+2|U_2| \leq  5^{\ell+2-i}$. By induction, there is a matching $M_{U'}$ {with $C(M_{U'})=C(M)$} which is disjoint from $U'$ and contains every edge in $M$ with colour outside $C_{i-2}$. Note that, therefore, $M_2\subset M_{U'}$.
Let $M_U=(M_{U'}\setminus M_2)\cup M_2'$. This replaces each edge $m\in M_2\subset M_{U'}$ with an edge of the same colour, $e_m$, which is disjoint from $U$, while keeping the edges in the matching independent. Thus, $C(M_U)=C(M_{U'})=C(M)$. As the edges moved had colour in $C_{i-1}$, $M_U$ contains every edge in $M_{U'}$ with colour outside  $C_{i-1}$, and thus every edge in $M$ with colour outside $C_{i-1}$.
\end{proof}

\begin{claim}\label{ClaimCiContainedInM}
For each $0\leq i\leq \ell$, all colours between $A_i$ and $B_i$ occur on $M$.
\end{claim}
\begin{proof}
Suppose that there is an edge $xy$ from $A_i$ to $B_i$ whose colour does not occur on $M$. Claim~\ref{ClaimExtendMatching} applied with $U=\{x,y\}$ gives a maximal rainbow matching $M_{U}$ in $G$ with $C(M_U)=C(M)$ and $V(M_U)\cap\{x,y\}=\emptyset$. Letting $M'=M_U+xy$ gives a rainbow matching in $G$ with size $|M|+1$, contradicting the maximality of $M$.
\end{proof}
Let $\mu=1/2\ell$, which is a non-zero decreasing function of $\eps$, so therefore $\mu\gg\eta$.

\begin{claim}\label{ClaimOnePlentifulColour}
Let $0\leq i\leq \ell$ and let $F$ be a set of at most $(1- \mu)|A_i|$ colours. Then, there is a colour $c\not\in F$ which occurs on at least $s$ disjoint edges between $A_i$ and $B_i$.
\end{claim}
\begin{proof} Let $t=\lceil 2ks/\mu\eps \rceil$, and note this is a function of $\e$ and $k$. Noting that $|A_i|\geq |A_0|\geq \e n$, take in $A_i$ disjoint sets $S_1, \dots, S_{t}$, each with size at least $\eta n$.
By the assumption of the lemma, there are at least $(1-\eta)|B_i|/k\geq (1-\eta)|A_i|$ colours between $S_j$ and $B_i$, for each $j\in[t]$. Therefore, there are at least $(1-\eta)|A_i|-|F|\geq (\mu-\eta)|A_i|\geq \mu \e n/2$ colours outside $F$ between $S_j$ and $B_i$ for each $j\in [t]$.
By Claim~\ref{ClaimCiContainedInM}, there are at most $n$ colours between $A_i$ and $B_i$. Therefore, there is a colour, $c$ say, outside $F$ which occurs between $B_i$ and at least $(\mu \e n/2)\cdot t/n\geq ks$ different sets $S_j$. As the colouring of $G$ is {locally} $k$-bounded and the sets $S_j$ are disjoint, there are thus at least $s$ disjoint edges with colour $c$ between $A_i$ and $B_i$.
\end{proof}

\begin{claim}\label{ClaimManyPlentifulColours}
For each $0\leq i\leq \ell$, $|C_i|\geq (1-\mu)|A_i|$.
\end{claim}
\begin{proof}
If $|C_i|< (1-\mu)|A_i|$, then, by Claim~\ref{ClaimOnePlentifulColour} with $F=C_i$, there is some colour $c\not\in C_i$ occuring at least $s$ times between $A_i$ and $B_i$. This colour should be in $C_i$, a contradiction.
\end{proof}

{Note that all colours of $C_{i-1}$ occur between $A_{i-1}$ and $B_{i-1}$, and so, by Claim~\ref{ClaimCiContainedInM},  occur on $M$. Since $M$ is rainbow,  the definition of $A_i$ implies that for each $i\in [\ell]$ we have $|A_i|=|A_0|+|C_{i-1}|$.} Combined with
Claim~\ref{ClaimManyPlentifulColours}, this gives $|A_i|\geq (1-\mu)|A_{i-1}|+|A_0|$
for each $i\in [\ell]$, and hence $|A_\ell| \geq \sum_{j=0}^{\ell-1} (1-\mu)^j |A_0|$. Note that, as $\mu=1/2\ell$, $(1-\mu)^j\geq 1/2$ for each $j\in[\ell]$, and therefore $|A_\ell|\geq \ell |A_0|/2$.
Thus, $|A_\ell|\geq \ell \cdot\e n/2>n$, a contradiction, so that the original matching $M$ must satisfy the property in the lemma.
\end{proof}

We now put this together with Lemma~\ref{randcol} to complete the proof of the main result of this section.

\begin{proof} [Proof of Lemma~\ref {LRainbowMatching}]

Let $\eta$  be a fixed constant not dependent on $n$ which satisfies $\e\gg\eta>0$. With probability $1-o(n^{-1})$, by Lemma~\ref{randcol} {applied with $\eta/2k$ for $\eps$}, for each $A\subset V(K_n)\setminus X$ and $B\subset X$ with $|A|,|B|\geq \eta pn/2k$,
there are at least $(1-\eta)|B|/k$ colours in $C$ between $A$ and $B$. With probability $1-o(n^{-1})$, by Lemma~\ref{chernoff}, $(1-\eta/2)pn\leq |X|\leq (1+\eta/2)pn$.

We claim that the property in the lemma holds.
Let $A\subset V(K_n)\setminus X$ with $|A|\leq pn/k$.
Add vertices to $A$ from $V(K_n)\setminus X$, or delete vertices from $A$, to get a set $A'$ with $|A'|=\lfloor(1-\eta/2)pn/k\rfloor=:m$ and $|A \setminus A'|\leq \eta pn/k\leq \varepsilon pn/2$. Let $X'$ be a subset of $X$ of size $km$.
Note that for any subsets $A''\subset A'$ and $B\subset X'$ with $|A''|,|B|\geq \eta m$, there are at least $(1-\eta)|B|/k$ colours in $C$ between $A''$ and $B$ (since $\eta m\geq  \eta pn/2k$).
Thus, by Lemma~\ref{LemmaRainbowMatching}, there is a  $C$-rainbow matching with at least $(1-\eps/4)m\geq |A'|-\eps pn/2$
edges between $A'$ and $X'$.
As $|A \setminus A'|\leq \varepsilon pn/2$ at least $|A|-\eps pn$ of the edges in this $C$-rainbow matching must lie between $A$ and $X$.
\end{proof}

We remark that, in the particular case where $K_n$ has an equal number of edges of each colour, Lemma~\ref{LRainbowMatching} can be proved from Lemma~\ref{boundrandcolour} and Corollary~\ref{randsetcor} using a standard implementation of the R\"odl nibble (see, for example,~\cite{alon2004probabilistic}).


\section{Star finding}\label{sec:starfind}
In this section we use a switching argument to find vertex-disjoint collectively-rainbow stars. The argument is based on one independently discovered by Woolbright \cite{woolbright78} and Brouwer, de Vries, and Wieringa~\cite{BVW78}. We then apply the resulting lemma through two corollaries to prove
a form applicable to our colourings. The resulting Corollary~\ref{kdisjstars} will be used to embed the part of our tree consisting of large stars.

\begin{lemma}\label{disjstars} Let $0<\e<{1/100}$ and $\ell\leq \eps^{2}n/2$. Let $G$ be an $n$-vertex graph with minimum degree at least $(1-\e)n$ which contains an independent set on the distinct vertices $v_1,\ldots,v_\ell$. Let $d_1,\ldots,d_\ell\geq 1$ be integers satisfying $\sum_{i\in [\ell]}d_i\leq (1-3\e)n$.
Let $G$ be edge-coloured so that no two edges touching the same vertex
$v_i$ are coloured the same.

 Then, $G$ contains vertex-disjoint stars $S_1,\ldots,S_\ell$ so that, for each $i\in [\ell]$, $S_i$ is a star rooted at $v_i$ with $d_i$ leaves, and $\cup_{i\in [\ell]}S_i$ is rainbow.
\end{lemma}
\begin{proof} Let $m=\lceil\eps^{-1}\rceil+1$.
Choose disjoint collectively-rainbow stars $S_1,\ldots,S_\ell$ in $G$ so that, for each $i\in [\ell]$, $S_i$ is rooted at $v_i$ and has at most $d_i+m$ leaves, and $\sum_{i\in [\ell]}|S_i|$ is maximised. Either these stars satisfy the lemma, or, without loss of generality, $S_1$ has fewer than $d_1$ leaves. Supose then, that $S_1$ has fewer than $d_1$ leaves.

Let $A_0=\emptyset$.
  Define the sets $C_1\subset C_2\subset \ldots \subset C_m\subset C(G)$ and $A_1\subset A_2\subset \ldots \subset A_m$ recursively by letting $C_j$ be the set of colours in $C(G)$ which do not appear on $(\cup_{i\in [\ell]}S_i)-A_{j-1}$ and letting $A_j=N_{C_{j}}(v_1)$.

\begin{claim}\label{gerald}
For each $s\in [m]$, $A_s\subset \cup_{i\in [\ell]}V(S_i)$.
\end{claim}
\begin{proof}

Suppose, for  contradiction, that there is some $u_1\in A_s$ with $u_1\not\in \cup_{i\in [\ell]}V(S_i)$. By the definition of $A_s$, we have $c(v_1u_1)=c_1$ for some colour $c_1\in C_s$.
Let $u_1, \dots, u_{j}\in V(G)$ and $c_1, \dots, c_{j}\in C(G)$ be longest sequences of distinct vertices and colours satisfying the following properties.
\begin{enumerate}[label = (\arabic{enumi})]
\item $c(v_1 u_k)=c_k$ for $k=1, \dots, j$.
\item $c(u_k v_r)=c_{k-1}$ for $k=2, \dots, j$, for some $v_r$ with $v_ru_k\in E(S_r)$.
\item For each $k=1, \dots, j$, $c_k\in C_{s+1-k}$.
\end{enumerate}
We claim that $j=s$, so that $c_{j}\in C_1$ and the colour $c_j$ does not appear on $\cup_{i\in [\ell]}S_i$. Suppose otherwise. Then, by (3) we have $c_{j}\in C_{s+1-j}$. By the definitions of $C_{s+1-j}$ and $C_1$, there is some $r$ and some edge $uv_r\in \bigcup_{i\in [\ell]}E(S_i)$ with  $u\in A_{s-j}$ and $c(uv_r)=c_{j}$.
By the definition of $A_{s-j}$, the edge $v_1u$ is present and has $c(v_1u)\in C_{s-j}$.
Notice that since $c_{j}\neq c_1, \dots, c_{j-1}$ (by (1)) and $\bigcup_{i=1}^{\ell} S_i$ is rainbow, this implies that $u\not\in\{u_1, \dots, u_{j}\}$ (by (2)).
Since the colours of the edges at $v_1$ are all distinct, we have $c(v_1u)\not\in\{c_1, \dots, c_{j}\}$ (using (1)).
Let $c_{j+1}=c(v_1u)$ and $u_{j+1}=u$.
This is a longer sequence satisfying (1) -- (3), contradicting the maximality of $j$.

Note that (3) implies that $j\leq m+1$. For each $i\geq 2$, let $S_i'=S_i-\{u_{2},\ldots,u_{j}\}$ and let $S_1'$ be the star formed from $S_1$ by adding the edges $v_1u_{1},\ldots,v_1u_j$.
In total, to get the stars $S_i'$, $i\in [\ell]$, we added an edge $v_1u_{j}$ of colour $c_{j}$, a colour not appearing on $\cup_{i\in[\ell]}S_i$, and, for each colour $c_i$, $i \in \{1,\ldots,j-1\}$ we removed some edge (adjacent to $u_{i+1}$) of that colour from some star $S_{i'}$ and then added an edge of that colour at $v_1$ ($v_1u_i$). Thus, the new stars are still collectively rainbow. As $S_1$ had fewer than $d_1$ leaves,
 $S_1'$ has at most $d_1+m$ leaves. Therefore, as $\sum_{i\in [\ell]}|S'_i|= 1+\sum_{i\in [\ell]}|S_i|$, this violates the choice of the stars
$S_i$, $i\in [\ell]$.
\end{proof}

For each $j\in [m]$, we have by Claim~\ref{gerald}, that
\[
|C(G)\setminus C_j|=|\cup_{i\in [\ell]}E(S_i)|-|A_{j-1}|\leq \sum_{i\in [\ell]}(d_i+m)-|A_{j-1}|\leq n - 3\eps n  +\ell m- |A_{j-1}|\leq n - 2\eps n  - |A_{j-1}|.
\]
As $v_1$ has at least $(1-\e)n$ neighbours, and thus is adjacent to edges of at least  $(1-\e)n$ different colours, we then have, for each $j\in [m]$,
\[
|A_{j}|\geq (1-\e)n-|C(G)\setminus C_j|\geq |A_{j-1}|+\eps n.
\]
Therefore, $|A_m|\geq m\cdot\eps  n >n$, a contradiction.
\end{proof}

Given a locally $k$-bounded colouring, we can create a new colouring based on this in order to apply Lemma~\ref{disjstars}, as follows.

\begin{corollary}\label{kboundstars} Let $0<\e<1/100$ and  $\ell\leq \eps^{2}n/2$. Let $G$ be an $n$-vertex graph with minimum degree at least $(1-\e)n$ which contains an independent set on the distinct vertices $v_1,\ldots,v_\ell$. Let $d_1,\ldots,d_\ell\geq 1$ be integers satisfying $\sum_{i\in [\ell]}d_i\leq (1-3\e)n$,
and suppose $G$ has a locally $k$-bounded edge-colouring.

 Then, $G$ contains disjoint stars $S_1,\ldots,S_\ell$ so that, for each $i\in [\ell]$, $S_i$ is a star rooted at $v_i$ with $d_i$ leaves, and $\cup_{i\in [\ell]}S_i$ has at most $k$ edges of each colour.
\end{corollary}
\begin{proof} For each $i\in [\ell]$, recolour the edges next to {$v_i$} in $G$ by, for each colour $c\in C(G)$, letting the edges next to $v_i$ with colour $c$ have distinct colours from $\{(c,1),\ldots,(c,k)\}$. In this new colouring, no two edges of the same colour can meet at some vertex $v_i$. Applying Lemma~\ref{disjstars}, find disjoint stars $S_1,\ldots,S_\ell$ so that, for each $i\in [\ell]$, $S_i$ is a star rooted at $v_i$ with
$d_i$ leaves, and the stars are collectively rainbow in the new colouring. Observing that each colour must appear at most $k$ times on $\cup_{i\in [\ell]}S_i$ in the original colouring completes the proof.
\end{proof}

We wish to find vertex-disjoint collectively-rainbow  stars in a locally $k$-bounded colouring. We show they exist within the stars found by Corollary~\ref{kboundstars}, as follows.

\begin{corollary}\label{kdisjstars} Let $0<\e<1/100$ and  $\ell\leq \eps^{2}n/2$. Let $G$ be an $n$-vertex graph with minimum degree at least $(1-\e)n$  which contains an independent set on the distinct vertices $v_1,\ldots,v_\ell$.  Let $d_1,\ldots,d_\ell\geq 1$ be integers satisfying $\sum_{i\in [\ell]}d_i\leq (1-3\e)n/k$,
and suppose $G$ has a locally $k$-bounded edge-colouring.

 Then, $G$ contains disjoint stars $S_1,\ldots,S_\ell$ so that, for each $i\in [\ell]$, $S_i$ is a star rooted at $v_i$ with $d_i$ leaves, and $\cup_{i\in [\ell]}S_i$ is rainbow.
\end{corollary}
\begin{proof}
By Corollary~\ref{kboundstars}, $G$ contains disjoint stars $S_1,\ldots,S_\ell$ so that, for each $i\in [\ell]$, $S_i$ is a star rooted at $v_i$ with $kd_i$ leaves, and $\cup_{i\in [\ell]}S_i$ has at most $k$ edges of each colour. Create an auxillary bipartite graph $H$ with vertex classes $[\ell]$ and $C(G)$, where there is an edge between $i\in [\ell]$
and $c\in C(G)$ exactly when there is an edge with colour $c$ in $S_i$.

Given any set $A\subset [\ell]$, there are $\sum_{i\in A}kd_i$ edges in $\cup_{i\in A}S_i$, and therefore at least $\sum_{i\in A}d_i$ different colours. Thus, for any set $A\subset [\ell]$, we have $|N_H(A)|\geq \sum_{i\in A}d_i$. As Hall's generalised matching condition (see, for example,~\cite[Chapter 3, Corollary 11]{bollobas2013modern}) is satisfied, there is a collection of disjoint stars $S_i'$, $i\in[\ell]$, in $H$, where, for each $i\in [\ell]$,
the star $S_i'$ is rooted at $i$ and has
$d_i$ leaves.

For each $i\in [\ell]$, pick a neighbour in $S_i$ for each colour in $S_i'$, and call the resulting star $S_i''$. The stars $S_i''$, $i\in [\ell]$, then satisfy the corollary.
\end{proof}


\section{Rainbow connecting paths}\label{conpaths}
In this section we prove that if we choose a random set $X$ and a random set of colours $\CC$, then, with high probability, we can connect any small collection of pairs of vertices
by collectively-rainbow vertex-disjoint paths of length 3, whose edges have colours in $\CC$ and whose intermediate vertices are in $X$.
Note that unlike some previous sections here we do assume that the random choices for $X$ and $\CC$ are completely independent.
First we need the following simple proposition.

\begin{prop}\label{randcolvertex}
Let $1/k,p,q\gg 1/n>0$ and suppose $K_n$ has a locally $k$-bounded colouring. Let $X\subset V(K_n)$ and $\CC\subset C(K_n)$ be subsets with each element chosen independently at random with probability $p$ and $q$ respectively. Almost surely, each vertex has at least $pq n/2$ colour-$C$ neighbours in $X$.
\end{prop}
\begin{proof}
For each $x\in V(K_n)$, let $d_x$ be the number of colour-$C$ neighbours in $X$. Note that $d_x$ is $k$-Lipschitz, and $\E d_x=pq(n-1)$. Thus, by Azuma's inequality (Lemma~\ref{Lemma_Azuma}) with $t=pqn/3$, we have $\P(d_x<pq n/2)\leq 2\exp(-p^2q^2n/9k^2)=o(n^{-1})$. Thus, $d_x\geq pq n/2$ for each $x\in V(K_n)$ with probability $1-o(1)$.
\end{proof}

\begin{lemma}\label{pathfinder}  Let $1/k, p\gg \mu \gg1/n>0$ and suppose $K_n$ has a {locally}  $k$-bounded colouring. Let $X\subset V(K_n)$ and $\CC\subset C(K_n)$ be subsets with each element chosen independently at random with probability $p$. Almost surely, for each pair of distinct vertices $u,v\in V(K_n)\setminus X$ there are at least
$\mu n$ internally vertex-disjoint collectively $C$-rainbow $u,v$-paths with length 3 and internal vertices in $X$.
\end{lemma}
\begin{proof} Create a random partition $C=C_1\cup C_2$ by assigning each element to a class uniformly at random. By Proposition~\ref{randcolvertex}, we almost surely have the following property.
\begin{enumerate}
\item[\textbf{Q1}] Each vertex has at least $100k^2\mu^{1/3} n$ colour $C_1$-neighbours in $X$.
\end{enumerate}

Note that ${10k\mu^{1/3}n}\geq n^{3/4}$.
Almost surely, by  Lemma~\ref{boundrandcolour}, we have the following property.
\begin{enumerate}
\item[\textbf{Q2}] Between every pair of disjoint subsets $A,B\subset V(K_n)$ with $|A|,|B|\geq 10k\mu^{1/3}n$ there are at least $p|A||B|/2\geq 4k\mu n^2$ colour-$C_2$ edges.
\end{enumerate}

Suppose then, for contradiction, there are some pair of distinct vertices $u,v\in V(K_n)$ and at most $\mu n$ internally vertex-disjoint collectively-$C$-rainbow $u,v$-paths with length 3 and internal vertices in $X$. Fixing a maximal set of such paths, $\mathcal{P}$, let $U\subset X$ be their set of internal vertices and $C'$ their set of edge colours.
Note that $|U|\leq 2\mu n$ and $|C'|\leq 3\mu n$.

By \textbf{Q1}, we have $|N_{C_1\setminus C'}(u,X\setminus U)|\geq 100k^2\mu^{1/3} n-2\mu n-3k\mu n\geq 10k\mu^{1/3} n$. Let $A\subset N_{C_1\setminus C'}(u,X\setminus U)$ satisfy $|A|=10k\mu^{1/3}n$,
and let $C''$ be the set of colours between $u$
and $A$. Using \textbf{Q1} again, we have
\[
|N_{C_1\setminus (C'\cup C'')}(v,X\setminus (U\cup A))|\geq 100k^2\mu^{1/3} n-2\mu n-3k\mu n-|A|-k|A|\geq 10k\mu^{1/3} n.
\]
Let $B\subset N_{C_1\setminus (C'\cup C'')}(v,X\setminus (U\cup A))$ satisfy $|B|=10k\mu^{1/3}n$. By \textbf{Q2},
there are at least $4k\mu n^2$ $C_2$-edges between $A$ and $B$, at most $kn\cdot|\CC'|\leq 3k\mu n^2$ of which have their colour in $C'$. Thus, there is some $x\in A$ and $y\in B$ so that $uxyv$ is a $(C\setminus C')$-rainbow path with internal vertices in $X\setminus U$. This contradicts the choice of $\mathcal{P}$.
\end{proof}

\section{Almost-spanning trees}\label{Section_Almost_Spanning_Trees}
We have now developed the tools that we will need to take an almost-spanning tree $T$ and embed it into a locally $k$-bounded edge-coloured $K_n$. We will do this using (carefully chosen) random partitions $V(K_n)=X_0\cup\ldots\cup X_\ell$ and $C(K_n)=C_0\cup\ldots\cup C_\ell$, with both $\ell$ and the distributions depending on $T$. The vertices in $X_0$ and the colours in $C_0$ are used at various points in the embedding to find small parts of the tree $T$. This will be possible greedily, but to ease the checking for these parts of the proof we give appropriate embedding results for small trees (or forests) in Section~\ref{greedy}. The two random partitions will not be independent, and depend both on $T$ and each other, and on certain almost-sure properties holding. This complicates our use of the probabilistic method to find the partition we use, and so we explain these particulars carefully in Section~\ref{probexplainer}. We then put together our proof of Theorem~\ref{almostspan}
in Section~\ref{finalproof}.

\subsection{Embedding small parts of $T$}\label{greedy}
When there are many spare colours and vertices, we can construct small rainbow trees, finish rainbow matchings and find collectively-rainbow vertex-disjoint connecting paths, as we do in the following three propositions.
\begin{prop}\label{embedT0}
Suppose we have an $m$-vertex tree $T$ and a graph $G$ with a locally $k$-bounded colouring in which $\delta(G)\geq 3k m$. Then, there is a rainbow copy of $T$ in $G$.
\end{prop}
\begin{proof}
Let $T'$ be a maximal subtree of $T$ which has a rainbow copy, $S'$ say, in $G$. Suppose, for contradiction, that $|S'|<m$ so that $S'$ has a vertex, $s$ say, to which a leaf can be appended to find a copy of a larger subtree than $T'$ in $T$. The edges of $S'$ have at most $m$ colours collectively, so $s$ has at most $km$ colour-$C(S')$ neighbours. Thus, $s$ must have at least $3km-|S'|-km>0$ colour-$(C(G)\setminus C(S'))$ neighbours in $V(G)\setminus V(S')$. Such a neighbour allows $S'$ to be extended to a larger rainbow copy of a subtree of $T$ than $S'$, a contradiction.
\end{proof}

\begin{prop}\label{embedTi}
 Suppose we have a graph $G$ with a locally $k$-bounded colouring and disjoint sets $X,Y,Z\subset V(G)$ and {disjoint} sets of colours $C,C'\subset C(G)$, such that there is a rainbow $C$-matching with at least $|X|-m$ edges from $X$ into $Y$, and each vertex in $G$ has at least ${2}km$ colour-$C'$ neighbours in $Z$.

Then, there is a $(C\cup C')$-rainbow matching with $|X|$ edges from $X$ into $Y\cup Z$ which uses at most $m$ colours in $C'$ and at most $m$ vertices in $Z$.
\end{prop}
\begin{proof}
Let $M_0$ be a $C$-rainbow matching with $|X|-m$ edges from $X$ into $Y$. Greedily, pick matchings $M_0\subset M_1\subset \ldots\subset  M_m$ so that, for each $i\in \{0,\ldots,m\}$, $M_i$ is a $(C\cup C')$-rainbow matching with $|X|-m+i$ edges from $X$ into $Y\cup Z$ which uses at most $i$ colours in $C'$ and at most $i$ vertices in $Z$.

Note this is possible, as, for each $i\in [m]$, if we have a satisfactory matching $M_{i-1}$, then choosing $x\in X\setminus V(M_{i-1})$ we have that $x$ has at most $k(i-1)$
colour-$(C'\cap C(M_{i-1}))$ neighbours in $Z$. Thus, $x$ has at least $2km-k(i-1)\geq km$ colour-$(C'\setminus C(M_{i-1}))$ neighbours in $Z$, at most $i-1\leq m-1$ of which can be in $V(M_{i-1})$.
Therefore, we can pick a  colour-$(C'\setminus C(M_{i-1}))$ neighbour $y$ of $x$ in $Z\setminus V(M_{i-1})$ and let $M_i=M_{i-1}\cup \{xy\}$.

Thus, we find a matching, $M_m$, as required.
\end{proof}

\begin{prop}\label{extendTj}
Suppose we have a graph $G$ with a locally $k$-bounded colouring containing the disjoint vertex sets $X=\{x_1,\ldots,x_m,x'_1,\ldots,x'_m\}$ and $Y$ such that, for each $i\in [m]$, there are at least $10m$ internally vertex-disjoint collectively-rainbow $x_i,x'_i$-paths of length three with interior vertices in $Y$.
Then, there is a vertex disjoint set of collectively rainbow $x_i,x'_i$-paths, $P_i$, $i\in [m]$, of length three with interior vertices in $Y$.
\end{prop}
\begin{proof}
Let $I\subset [m]$ be a maximal subset for which there are vertex-disjoint collectively-rainbow $x_i,x'_i$-paths, $P_i$, $i\in I$, of length three with interior vertices in $Y$. Suppose, for contradiction, that $I\neq [m]$, and pick $j\in [m]\setminus I$.

Consider a collection $\mathcal{Q}$ of $10m$ internally vertex-disjoint collectively-rainbow $x_j,x'_j$-paths of length three with interior vertices in $Y$.
Let $P=\cup_{i\in I}P_i$. Note that $|C(P)|=3|I|\leq 3m$ and $|P|\leq 4m$. Thus, there can be at most $7m$ paths in $\mathcal{Q}$ with an edge with colour in $C(P)$ or a vertex in $V(P)$. Therefore, we can pick a path $P_j\in \mathcal{Q}$ so that $P_i$, $i\in I\cup\{j\}$, are vertex-disjoint collectively-rainbow $x_i,x'_i$-paths of length three with interior vertices in $Y$, a contradiction.
\end{proof}

\subsection{Dependence and the probabilistic method}\label{probexplainer}
To find a rainbow copy of an almost-spanning tree $T$, we will split the tree into pieces, find a random partition of the vertices and colours of $K_n$, and show that the properties we need to embed the tree (\ref{R1}-\ref{Q1} as listed later) almost surely hold. Thus, there will be some partitions for which these properties hold, and using these we can then embed the tree $T$. Our implementation of the probabilistic method here is complicated by the dependence of some of the random sets in the partitions on each other, and furthermore \ref{Q0} may only hold if \ref{R1}-\ref{R3} hold, while \ref{Q1} holds only if the other properties all hold. Therefore, we discuss this here in detail to clarify this aspect of our proof, and explicitly give the simple formalities we later pass over.

In total, for some integer $\ell= O(\log^{10}n)$ and probabilities $p_0,\ldots,p_\ell$ depending on $T$ we will pick random partitions
\[
V(K_n)=X_0\cup\ldots\cup X_\ell\;\;\text{ and }\;\; C(K_n)=C_0\cup\ldots\cup C_\ell
\]
such that
\begin{enumerate}[label = (\roman{enumi})]
\item\label{frost1} for each $0\leq i\leq \ell$, $x\in V(K_n)$ and $c\in C(K_n)$, we have $\P(x\in X_i)=\P(c\in C_i)=p_i$,
\item\label{frost2} the choice of the set containing each $x\in V(K_n)$ is made independently of the choice for each other vertex in $V(K_n)$,
\item\label{frost3} the choice of the set containing each $c\in C(K_n)$ is made independently of the choice for each other vertex in $C(K_n)$, and
\item\label{frost4} the choice of $X_0$ is made independently of the choice of $C_0$.
\end{enumerate}
We do this by first selecting vertices for $X_0$ and colours for $C_0$ independently at random with probability $p_0$. Thus, \ref{frost4} holds and \ref{frost1}-\ref{frost3} hold for $i=0$. The properties \ref{R1}-\ref{R3} will depend only on $X_0$ and $C_0$, and will almost surely hold.

If \ref{R1}-\ref{R3} hold, we find a $C_0$-rainbow copy $S_0$ of part of the tree $T$ using vertices in $X_0$, and then, depending on $S_0$, pair off some vertices in $V(K_n)\setminus X_0$ with colours in $C(K_n)\setminus C_0$. Formally, if \ref{R1}-\ref{R3} do not hold then we let $S_0=\emptyset$ and take no such pairs. We then take a vertex set $X_1$ by selecting vertices in $V(K_n)\setminus X_0$ independently at random with probability $p_1/(1-p_0)$.
Each $x\in V(K_n)$ thus appears in $X_1$ with probability $p_1$ independently of each other vertex. Almost surely, if \ref{R1}-\ref{R3} hold, then \ref{Q0} will hold (a property depending only on $X_0$, $C_0$ and $X_1$).
If any of \ref{R1}-\ref{R3} do not hold then we say \ref{Q0} does not hold.

{Now we define the set $C_1$ which is disjoint from $C_0$}. For each colour paired with a vertex, we take it in $C_1$ precisely if  its paired vertex is in $X_1$. For each colour not paired with a vertex, we take it in $C_1$ independently at random with probability $p_1/(1-p_0)$.  Each colour is paired to at most one vertex, and each vertex is paired to at most one colour, so colours paired with a vertex appear in $C_1$ uniformly at random with probability $p_1/(1-p_0)$.
Thus, whether it is paired to a vertex or not, each colour in $C(K_n)\setminus C_0$ appears in $C_1$ independently at random with probability  $p_1/(1-p_0)$. Therefore, each colour in $C(K_n)$ appears in $C_1$ independently at random with probability $p_1$, completing the requirements of \ref{frost1}--\ref{frost3}
for $i=1$.

We then take random partitions
\[
V(K_n)\setminus(X_0\cup X_1)=X_2\cup\ldots\cup X_\ell\;\text{ and }\; C(K_n)\setminus(C_0\cup C_1)=C_2\cup\ldots\cup C_\ell,
\]
by selecting the location of each $x\in V(K_n)\setminus(X_0\cup X_1)$ and $c\in C(K_n)\setminus(C_0\cup C_1)$ independently at random so that, for each $2\leq i\leq \ell$, $\P(x\in X_i)=\P(c\in C_i)=p_i/(1-p_0-p_1)$.

We can simply observe that, for each $x\in V(K_n)$ and $i\in [\ell]$, $\P(x\in X_i)=p_i$, and, furthermore, this is independent of the location of any other vertex, so that \ref{frost1} holds for each vertex, and \ref{frost2} holds. Similarly, we can see that
 \ref{frost1} holds for each colour, and \ref{frost3} holds.

If \ref{R1}--\ref{Q0} hold, then our fifth property, \ref{Q1}, will hold for each $2\leq i\leq \ell$ with probability $1-o(\ell n^{-1})=1-o(1)$ by applying to $X_i$ and $C_i$ a result (Lemma~\ref{LRainbowMatching}) which needs precisely that \ref{frost1}-\ref{frost3} hold for $i$ but asks for no independence between $C_i$ and $X_i$.
If any of  \ref{R1}--\ref{Q0} do not hold, then we say that \ref{Q1} does not hold.

Finally, as $\P(\ref{R1}-\ref{R3}\text{ hold})=1-o(1)$,
$\P(\ref{Q0}\text{ holds}|\ref{R1}-\ref{R3}\text{ hold})=1-o(1)$ and $\P(\ref{Q1}\text{ holds}|\ref{R1}-\ref{Q0}\text{ hold})=1-o(1)$, we have that $\P(\ref{R1}-\ref{Q1}\text{ hold})=1-o(1)$.
Thus, there must be some partitions $V(K_n)=X_0\cup\ldots\cup X_\ell$ and $C(K_n)=C_0\cup\ldots\cup C_\ell$ such that \ref{R1}--\ref{Q1} hold with the copy $S_0$ of part of $T$, and we will then complete the copy of $T$ using such partitions, starting with the tree $S_0$.

\subsection{Proof of Theorem~\ref{almostspan}}\label{finalproof}

\begin{proof}[Proof of Theorem~\ref{almostspan}] Let $\mu$ satisfy $\e,1/k\gg \mu\gg 1/n>0$ and let $D=\lceil\log^{10}n\rceil$. Let $T$ be a tree with at most $(1-\e)n/k$ vertices and let $K_n$ have a locally $k$-bounded colouring.

\medskip

\textbf{Split $T$.} Using Lemma~\ref{decomp}, find integers $\ell\leq 10^4 D\mu^{-2}$ and $j\in[\ell]$ and a sequence of subgraphs $T_0\subset T_1\subset \ldots \subset T_\ell=T$ such that, for each $i\in [\ell]\setminus \{1,j\}$, $T_{i}$ is formed from $T_{i-1}$ by adding non-neighbouring leaves,
$T_j$ is formed from $T_{j-1}$ by adding at most $\mu n$ vertex disjoint bare paths with length $3$, $T_1$ is formed from $T_0$ by adding vertex disjoint stars with at least $D$ leaves each, and $|T_0|\leq \mu n$.

\medskip

\textbf{Choose `greedy' vertices and colours.} Pick random subsets $X_0\subset V(K_n)$ and $C_0\subset C(K_n)$ by selecting each element uniformly at random with probability $p_0:=\e/200k$. By Proposition~\ref{randcolvertex} and Lemma~\ref{pathfinder},  we almost surely have the following properties (For more details on what happens if they, or any subsequent properties, do not hold, see Section~\ref{probexplainer}).

\medskip

\begin{enumerate}[label = \textbf{R\arabic{enumi}}]
\item Each vertex in $V(K_n)$ has at least $10k\mu n$ colour-$\CC_0$ neighbours in $X_0$.
\label{R1}
\item For each pair of vertices $u,v\in V(K_n)$, there are at least $20\mu n$ internally vertex-disjoint collectively  $\CC_0$-rainbow $u,v$-paths with length 3 and interior vertices in $X_0$.
\label{R2}
\end{enumerate}
Furthermore, by Lemma~\ref{chernoff}, almost surely we have $|X_0|,|C_0|\leq \eps n/100k$, and hence any vertex is contained in at most $\e n/100$ $C_0$-edges. Thus, the following almost surely holds.
\begin{enumerate}[label = \textbf{R3}]
\item If $G$ is the subgraph of $K_n$ of the edges with colour in $C(K_n)\setminus C_0$, with any edges inside $X_0$ removed, then $\delta(G)\geq (1-\e/50)n$.\label{R3}
\end{enumerate}

\medskip

\textbf{Embed $T_0$ and find rainbow stars.} Using {\ref{R1} and} Proposition~\ref{embedT0}, pick a $\CC_0$-rainbow copy, $S_0$ say, of $T_0$ in $X_0$. Then, for the appropriate integers $m\leq n/D$ and  $d_1, \dots, d_m\geq D$, let $v_1,\ldots,v_m\in V(S_0)$ be such that $S_0$ can be made into a copy of $T_1$ by adding $d_i$ new leaves at $v_i$, for each $i\in [m]$.
Let $d=\sum_{i\in[m]}d_i=|T_1|-|T_0|\leq (1-\e)n/k$.
For each $i\in[m]$, let $n_i=\lceil (1-\e/8)nd_i/kd \rceil$. Note that $\sum_{i\in [m]} n_i\leq (1-\e/8)n/k+m\leq (1-\e/10)n/k$. Using \ref{R3} and Corollary~\ref{kdisjstars}, find disjoint subsets
$Y_i\subset V(K_n)\setminus X_0$, $i\in [m]$, so that $|Y_i|=n_i$, and $\{v_iy:i\in[m],y\in Y_i\}$ is $(\CC(K_n)\setminus C_0)$-rainbow.

{For each vertex $x$ in some set $Y_i$, pair $x$ with the colour $c$ of $v_ix$, noting that, as $\{v_iy:i\in[m],y\in Y_i\}$ is rainbow, each colour or vertex is in at most 1 pair. We will later choose a random vertex set $X_1$ and a random colour set $C_1$ so that, for such a pairing, $x\in X_1$ if and only if $c\in C_1$. For more details on this, see Section~\ref{probexplainer}.}

\medskip

\textbf{Choose probabilities $p_i$.}
For each $i\in [\ell]$, let $m_i=|T_i|-|T_{i-1}|$, and note that $m_1=d$. For each $i\in [\ell-1]$, let
\begin{equation}\label{pidefn}
p_i=(1+\e/4)km_i/n+\eps/4\ell\geq n^{-1/10^{4}},
\end{equation}
where the inequality follows as $\ell\leq 10^4D\mu^{-2}=O(\log^{10}n)$. Let
{
\begin{align}
p_\ell&=1-p_0-\sum_{i\in [\ell-1]}p_i= 1-\frac{\eps}{200k}- (1+\e/4)k\cdot\frac{|T|-m_\ell-|T_0|}{n}-\frac{\eps(\ell-1)}{4\ell}\nonumber\\
&\geq 1- (1+\e/4)k\cdot \frac{(1-\eps)n/k-m_\ell}{n}-\frac{\eps}4 = (1+\e/4)km_\ell/n+\eps\left(\frac1{2}+\frac{\eps}{4}\right)\geq n^{-1/10^{4}}.\label{pidefn2}
\end{align}}

\medskip

\textbf{Choose $X_1$.}
Pick $X_1\subset V(K_n)\setminus X_0$ by including each vertex independently at random with probability $p_1/(1-p_0)$.
Recall that $m_1=d$. By \eqref{pidefn}, we have, for each $i\in [m]$, that
\begin{align*}
p_{1}n_i&\geq (1+\e/4)km_1/n\cdot (1-\e/8)nd_i/kd\\
&=(1+\e/4)\cdot (1-\e/8)d_i \geq (1+\e/16)d_i\geq \log^{10}n.
\end{align*}
Thus, by Lemma~\ref{chernoff}, for each $i\in [m]$, $\P(|X_1\cap Y_i|\geq d_i)=\exp(-\Omega(\e^2\log^{10}n))=o(n^{-1})$. Thus, almost surely, the following property holds.
\begin{enumerate}[label = \textbf{R4}]
\item \label{Q0} For each $i\in [m]$,  $|X_1\cap Y_i|\geq d_i$.
\end{enumerate}
Note, for later, that each vertex $x\in V(K_n)$ appears in $X_1$ independently at random with probability~$p_1$.

\medskip

\textbf{Choose $C_1$.}
{Let $C^{\mathrm{paired}}$ be the set of colours which appear between $v_i$ and $Y_i$ for some $i\in [m]$, and let $C^{\mathrm{unpaired}}=C\setminus C^{\mathrm{paired}}$ be the set of colours which never appear between any $v_i$ and $Y_i$. We define a random set of colours $\CC_1$ as follows. For any colour $c\in C^{\mathrm{paired}}$, $c$ is included in $\CC_1$ whenever the vertex paired with $c$ is in $X_1$,
i.e.\ when $c$ appears between $v_i$ and $X_1\cap Y_i$ for some $i\in[m]$.  For any colour $c\in C^{\mathrm{unpaired}}\setminus C_0$, $c$ is included in $\CC_1$ independently at random with probability $p_1/(1-p_0)$. {Thus, $C_1$ contains each colour paired with a vertex in $X_1$ and each unpaired colour {outside $C_0$} is included uniformly at random.} Thus, each colour appears in $C_1$ independently at random with probability $p_1$.}

\medskip

\textbf{Choose a random vertex partition.}
Randomly partition $V(K_n)\setminus (X_0\cup X_1)$ as $
X_2\cup \ldots \cup X_\ell$
so that, for each $x\in V(K_n)\setminus (X_0\cup X_1)$, the class of $x$ is chosen independently at random with $\P(x\in X_i)=p_i/(1-p_0-p_1)$ for each $2\leq i\leq \ell$. Note that, for each $i\in \{0,1,\ldots,\ell\}$, each $x\in V(K_n)$ appears in $X_i$ independently at random with probability $p_i$, and the location of each vertex in $V(K_n)$ is independent of the location of all the other vertices.

\medskip

\textbf{Choose a random colour partition.}
Randomly partition $C(K_n)\setminus (C_0\cup C_1)$ as $
C_2\cup \ldots \cup C_\ell$ so that, for each $c\in \CC\setminus (\CC_0\cup \CC_1)$, the class of $c$ is chosen independently at random with $\P(c\in \CC_i)=p_i/(1-p_1-p_0)$ for each $2\leq i\leq \ell$.
Note that, for each $0\leq i\leq \ell$, each colour $c\in \CC(K_n)$ appears in $\CC_i$ independently at random with probability  $p_i$, and the location of each colour in $C(K_n)$ is independent of the location of all the other colours.

\medskip

\textbf{Rainbow matching properties.}
 Note that, by \eqref{pidefn} and \eqref{pidefn2}, $m_i\leq p_in/k$ for each $i\in [\ell]$. Therefore, from the properties of the random partitions of $C(K_n)$ and $V(K_n)$, and Lemma~\ref{LRainbowMatching}, we almost surely have the following property.

\begin{enumerate}[label = \textbf{R5}]
\item \label{Q1} For each $i\in [\ell]$ and subset $A\subset V(K_n)\setminus X_i$ with $|A|=m_i\leq p_in/k$ there is a $C_i$-rainbow matching with at least $m_i- \mu p_in$ edges from $A$ into $X_i$.
\end{enumerate}

As detailed in Section~\ref{probexplainer}, we can thus fix deterministic partitions of $V(K_n)$ and $C(K_n)$, and the copy $S_0$ of $T_0$, for which \ref{R1}--\ref{Q1} hold.

\medskip

\textbf{Extend to cover $T_1$.} For each $i\in [m]$, use \ref{Q0} to add $d_i$ leaves from $X_1\cap Y_i$ to $v_i$ in $S_0$ and call the resulting graph $S_1$. Note that these additions add leaves from $X_1$ using colours from $C_1$. Thus, $S_1\subset K_n[X_0\cup X_1]$ is a $(C_0\cup C_1)$-rainbow copy of $T_1$ with at most $\mu n$ colours in $C_0$ and at most $\mu n$ vertices in $X_0$.

\medskip

\textbf{Iteratively, extend to cover $T_2,\ldots,T_{j-1}$.}  Iteratively, for each $2\leq i\leq j-1$, extend $S_{i-1}$ to $S_i\subset K_n[X_0\cup\ldots\cup X_i]$,
a $(C_0\cup \ldots \cup C_i)$-rainbow copy of $T_i$ with $|C(S_i)\cap C_0|\leq \mu n+\sum_{i'=2}^i \mu p_{i'} n$ and $|V(S_i)\cap X_0|\leq \mu n+\sum_{i'=2}^i \mu p_{i'} n$ (so that, certainly, $|C(S_i)\cap C_0|\leq 2\mu n$ and  $|V(S_i)\cap X_0|\leq {2}\mu n$). Note that $T_i$ is obtained from $T_{i-1}$ by adding a matching (i.e.\ a collection of non-neighbouring leaves). Let $A_i \subset S_{i-1}$ be the vertex set to which we need to attach the edges of the matching. Then we can first apply \ref{Q1} to sets $A_i, X_i$ and the set of colours $C_i$ to find a matching
of size $|A_i|-\mu p_i n$ and then use \ref{R1} and Proposition~\ref{embedTi} with $G\subset K_n$ as the graph of colour-$((C_0\setminus C({S_{i-1}}))\cup C_i)$ edges,  $C=C_i$, $C'=C_0\setminus C(S_{i-1})$, $X=A_i, Y=X_i$ and $Z=X_0\setminus V({S_{i-1}})$ to find a rainbow matching covering the whole set $A_i$.

\medskip

\textbf{Extend to cover $T_j$.} Using new vertices in $X_0$ and new colours in $C_0$, extend this to $S_j\subset K_n[X_0\cup\ldots\cup X_j]$,
a $(C_0\cup \ldots \cup C_j)$-rainbow copy of $T_j$ with $|C(S_j)\cap C_0|\leq 4\mu n+\sum_{i=2}^j \mu p_{i}n$ and $|V(S_j)\cap X_0|\leq 3\mu n+\sum_{i=2}^j \mu p_{i}n$.
Note that, per path, we are using 3 additional colours from $C_0$ and 2 additional vertices from $X_0$, which explains the constants $4$ and $3$ in the last two inequalities.
This is possible by \ref{R2}, and Proposition~\ref{extendTj} applied with $G\subset K_n$ as the graph of colour-$(C_0\setminus C(S_{j-1}))$ edges and $Y=X_0\setminus V({S_{j-1}})$.

\medskip

\textbf{Iteratively, extend to cover $T_{j+1},\ldots,T_{\ell}$.} Finally, for each  $i\in \{j+1,\ldots, \ell\}$, use \ref{R1}, \ref{Q1} and Proposition~\ref{embedTi} as before to extend $S_{i-1}$ to $S_i\subset K_n[X_0\cup\ldots\cup X_i]$,
a $(C_0\cup \ldots \cup C_i)$-rainbow copy of $T_i$  with at most $4\mu n+\sum_{i'=2}^i \mu p_{i'} n$ colours in $C_0$ and at most $3\mu n+\sum_{i'=2}^i \mu p_{i'} n$ vertices in $X_0$. When this is finished, we have a rainbow copy of $T_\ell=T$, as required.
\end{proof}

\section{Concluding Remarks}
\begin{itemize}
\item
Our main theorem shows that properly coloured graphs have rainbow copies of every tree on $n-o(n)$ vertices. The most natural open problem is to ask how small the $o(n)$ term can be made. Note that we cannot take $o(n)=0$ here, since there are proper colourings of $K_n$ which do not have a rainbow copy of every $n$-vertex tree (see \cite{maamoun1984problem, BenzingTrees}). This shows that the error term in Theorem~\ref{almostspan} cannot be eliminated. However, we conjecture that it can be reduced to a constant.
\begin{conjecture}
There is a constant $C$ so that every properly coloured $K_n$ has a rainbow copy of every tree on $n-C$ vertices.
\end{conjecture}
\noindent If true, this conjecture would lead to corresponding improvements to our applications for the conjectures of Graham-Sloane and Gronau-Mullin-Rosa.

\item
This paper gives a unified approach for attacking three conjectures about graph decomposition and labelling. In particular, we overcame one of the most significant barriers to progress towards solving  these conjectures, i.e., we embedded all trees, rather than just trees with bounded degree. This ``bounded degree'' barrier exists in many other results about finding trees and more general subgraphs. Therefore, we expect that our methods might be useful to attack additional open problems in graph theory. One particular candidate is the Gy\'arf\'as Tree Packing Conjecture \cite{gyarfas1978packing}.
\begin{conjecture}[Gy\'arf\'as]\label{ConjectureGyarfas}
Let $T_1, \dots, T_{n-1}$ be trees with $|T_i|=i$ for each $i\in [n-1]$. The edges of $K_n$ can be decomposed into $n-1$ trees which are isomorphic to  $T_1, \dots, T_{n-1}$ respectively.
\end{conjecture}
\noindent
Notice the strong parallels between this conjecture and Ringel's Conjecture --- both conjectures concern the existence of decompositions of the complete graph into trees. Research on both these conjectures has progressed in parallel. That is, Conjecture~\ref{ConjectureGyarfas} has also recently been proved for bounded degree trees by   Joos, Kim, K{\"u}hn and Osthus~\cite{joos2016optimal}  (see also  \cite{bottcher2016approximate, messuti2016packing, ferber2017packing, kim2016blow} for other results). There is no asymptotic version of Conjecture~\ref{ConjectureGyarfas} known for trees with arbitrary degrees, so it would be interesting to see whether our methods can be used here.
\end{itemize}

\section*{Acknowledgement}Parts of this work were carried out when the first author visited the Institute for
Mathematical Research (FIM) of ETH Zurich. We would like to thank
FIM for its hospitality and for creating a stimulating research environment.

\bibliographystyle{abbrv}
\bibliography{rainbowtrees}

\begin{thebibliography}{10}

\bibitem{adamaszek2016almost}
A.~Adamaszek, P.~Allen, C.~Grosu, and J.~Hladk{\'y}.
\newblock Almost all trees are almost graceful.
\newblock {\em arXiv preprint arXiv:1608.01577}, 2016.

\bibitem{akbari2007rainbow}
S.~Akbari, O.~Etesami, H.~Mahini, and M.~Mahmoody.
\newblock On rainbow cycles in edge colored complete graphs.
\newblock {\em Australasian Journal of Combinatorics}, 37:33, 2007.

\bibitem{allen2017packing}
P.~Allen, J.~B{\"o}ttcher, J.~Hladk{\'y}, and D.~Piguet.
\newblock Packing degenerate graphs.
\newblock {\em arXiv preprint arXiv:1711.04869}, 2017.

\bibitem{alon2016random}
N.~Alon, A.~Pokrovskiy, and B.~Sudakov.
\newblock Random subgraphs of properly edge-coloured complete graphs and long
  rainbow cycles.
\newblock {\em Israel Journal of Mathematics}, 222(1):317--331, 2017.

\bibitem{alon2004probabilistic}
N.~Alon and J.~Spencer.
\newblock {\em \normalfont\bfseries The probabilistic method}.
\newblock John Wiley \& Sons, 2004.

\bibitem{andersen1989hamilton}
L.~Andersen.
\newblock Hamilton circuits with many colours in properly edge-coloured
  complete graphs.
\newblock {\em Mathematica {S}candinavica}, pages 5--14, 1989.

\bibitem{BenzingTrees}
F.~Benzing, A.~Pokrovskiy, and B.~Sudakov.
\newblock Long directed rainbow cycles and rainbow spanning trees.
\newblock {\em arXiv:1711.03772}, 2017.

\bibitem{bollobas2013modern}
B.~Bollob{\'a}s.
\newblock {\em {\normalfont\bfseries{Modern graph theory}}}, volume 184.
\newblock Springer, 2002.

\bibitem{bottcher2016approximate}
J.~B{\"o}ttcher, J.~Hladk{\'y}, D.~Piguet, and A.~Taraz.
\newblock An approximate version of the tree packing conjecture.
\newblock {\em Israel Journal of Mathematics}, 211(1):391--446, 2016.

\bibitem{BVW78}
A.~Brouwer, A.~de~Vries, , and R.~Wieringa.
\newblock A lower bound for the length of partial transversals in a latin
  square.
\newblock {\em Nieuw Archief Voor Wiskunde}, 26(2):330, 1978.

\bibitem{Brualdi}
R.~A. Brualdi and H.~J. Ryser.
\newblock {\em \normalfont\bfseries Combinatorial matrix theory}.
\newblock Cambridge University Press, 1991.

\bibitem{chen2015long}
H.~Chen and X.~Li.
\newblock Long rainbow path in properly edge-colored complete graphs.
\newblock {\em arXiv preprint arXiv:1503.04516}, 2015.

\bibitem{dinitz1992contemporary}
J.~H. Dinitz and D.~R. Stinson.
\newblock {\em \normalfont{\bfseries Contemporary design theory: A collection
  of surveys}}, volume~26.
\newblock John Wiley \& Sons, 1992.

\bibitem{Euler}
L.~Euler.
\newblock Recherches sur une nouvelle esp\'ece de quarr\'es magiques.
\newblock {\em Verh. Zeeuwsch. Gennot. Weten. Vliss.}, 9:85--239, 1782.

\bibitem{ferber2017packing}
A.~Ferber, C.~Lee, and F.~Mousset.
\newblock Packing spanning graphs from separable families.
\newblock {\em Israel Journal of Mathematics}, 219(2):959--982, 2017.

\bibitem{ferber2016packing}
A.~Ferber and W.~Samotij.
\newblock Packing trees of unbounded degrees in random graphs.
\newblock {\em arXiv preprint arXiv:1607.07342}, 2016.

\bibitem{gallian2009dynamic}
J.~A. Gallian.
\newblock A dynamic survey of graph labeling.
\newblock {\em The electronic journal of combinatorics}, 16(6):1--219, 2009.

\bibitem{gebauer2012rainbow}
H.~Gebauer and F.~Mousset.
\newblock On rainbow cycles and paths.
\newblock {\em arXiv preprint arXiv:1207.0840}, 2012.

\bibitem{GS80}
R.~Graham and N.~Sloane.
\newblock On additive bases and harmonious graphs.
\newblock {\em SIAM Journal on Algebraic Discrete Methods}, 1(4):382--404,
  1980.

\bibitem{gronau1997orthogonal}
H.-D. Gronau, R.~C. Mullin, and A.~Rosa.
\newblock Orthogonal double covers of complete graphs by trees.
\newblock {\em Graphs and Combinatorics}, 13(3):251--262, 1997.

\bibitem{gyarfas1978packing}
A.~Gy{\'a}rf{\'a}s and J.~Lehel.
\newblock Packing trees of different order into {$K_n$}.
\newblock In {\em Combinatorics (Proc. Fifth Hungarian Colloq., Keszthely,
  1976)}, volume~1, pages 463--469, 1978.

\bibitem{gyarfas2010rainbow}
A.~Gy\'arf\'as and M.~Mhalla.
\newblock Rainbow and orthogonal paths in factorizations of ${K}_n$.
\newblock {\em Journal of Combinatorial Designs}, 18(3):167--176, 2010.

\bibitem{gyarfas2011long}
A.~Gy{\'a}rf{\'a}s, M.~Ruszink{\'o}, G.~S{\'a}rk{\"o}zy, and R.~Schelp.
\newblock Long rainbow cycles in proper edge-colorings of complete graphs.
\newblock {\em Australasian Journal of Combinatorics}, 50:45--53, 2011.

\bibitem{hahn1980jeu}
G.~Hahn.
\newblock Un jeu de coloration.
\newblock In {\em Regards sur la theorie des graphes, Actes du Colloque de
  Cerisy}, volume~12, page~18, 1980.

\bibitem{hughes1978biplanes}
D.~Hughes.
\newblock Biplanes and semi-biplanes.
\newblock {\em Combinatorial Mathematics}, pages 55--58, 1978.

\bibitem{janson2011random}
S.~Janson, T.~{\L}{uczak}, and A.~Rucinski.
\newblock {\em \normalfont\bfseries Random graphs}.
\newblock Wiley, 2011.

\bibitem{joos2016optimal}
F.~Joos, J.~Kim, D.~K{\"u}hn, and D.~Osthus.
\newblock Optimal packings of bounded degree trees.
\newblock {\em J. European Math. Soc., to appear}, 2018.

\bibitem{keedwell2015latin}
A.~Keedwell and J.~D{\'e}nes.
\newblock {\em \normalfont\bfseries Latin squares and their applications}.
\newblock Elsevier Science, 2015.

\bibitem{kim2016blow}
J.~Kim, D.~K{\"u}hn, D.~Osthus, and M.~Tyomkyn.
\newblock A blow-up lemma for approximate decompositions.
\newblock {\em Trans. Amer. Math. Soc., to appear}, 2018.

\bibitem{leck1997orthogonal}
U.~Leck and V.~Leck.
\newblock On orthogonal double covers by trees.
\newblock {\em Journal of Combinatorial Designs}, 5(6):433--441, 1997.

\bibitem{maamoun1984problem}
M.~Maamoun and H.~Meyniel.
\newblock On a problem of {G}. {H}ahn about coloured {H}amiltonian paths in
  ${K}_{2^t}$.
\newblock {\em Discrete mathematics}, 51(2):213--214, 1984.

\bibitem{messuti2016packing}
S.~Messuti, V.~R{\"o}dl, and M.~Schacht.
\newblock Packing minor-closed families of graphs into complete graphs.
\newblock {\em Journal of Combinatorial Theory, Series B}, 119:245--265, 2016.

\bibitem{ringel1963theory}
G.~Ringel.
\newblock Theory of graphs and its applications.
\newblock In {\em Proceedings of the Symposium Smolenice}, 1963.

\bibitem{rosa1966certain}
A.~R\'{o}sa.
\newblock On certain valuations of the vertices of a graph.
\newblock In {\em Theory of Graphs (Internat. Symposium, Rome)}, pages
  349--355, 1966.

\bibitem{ErdosRado}
S.~K. Stein.
\newblock A combinatorial theorem.
\newblock {\em J. London Math. Soc.}, 25:249--255, 1950.

\bibitem{Stein}
S.~K. Stein.
\newblock Transversals of {L}atin squares and their generalizations.
\newblock {\em Pacific J. Math.}, 59:567--575, 1975.

\bibitem{WanlessSurvey}
I.~M. Wanless.
\newblock Transversals in {L}atin squares: A survey.
\newblock {\em Surveys in Combinatorics}, 2011.

\bibitem{woolbright78}
D.~Woolbright.
\newblock An {$n\times n$} latin square has a transversal with at least $n-
  \sqrt{n}$ distinct symbols.
\newblock {\em Journal of Combinatorial Theory, Series A}, 24(2):235--237,
  1978.

\bibitem{wozniak2004packing}
M.~Wo{\'z}niak.
\newblock Packing of graphs and permutations---a survey.
\newblock {\em Discrete Mathematics}, 276:379--391, 2004.

\bibitem{yap1988packing}
H.~Yap.
\newblock Packing of graphs---a survey.
\newblock {\em Annals of Discrete Mathematics}, 38:395--404, 1988.

\bibitem{zak2009harmonious}
A.~{\.Z}ak.
\newblock Harmonious order of graphs.
\newblock {\em Discrete Mathematics}, 309(20):6055--6064, 2009.

\end{thebibliography}

\end{document}